\documentclass[leqno,twoside, 10pt]{amsart}
\usepackage{amssymb}
\usepackage{latexsym,esint}
\usepackage{color}

\def\endproof{\hspace*{\fill}\mbox{\ \rule{.1in}{.1in}}\medskip }

\theoremstyle{plain}

\numberwithin{equation}{section} \numberwithin{figure}{section}

\newtheorem{theorem}{Theorem}[section]
\newtheorem{lemma}[theorem]{Lemma}

\begin{document}
\title[Convergence of equilibria for incompressible plates]
{Convergence of equilibria \\ for incompressible elastic plates \\
in the von K\'arm\'an regime}
\author{Marta Lewicka}
\address{Marta Lewicka,  University of Pittsburgh, Department of Mathematics,
301 Thackeray Hall, Pittsburgh, PA 15260, USA }
\email{lewicka@math.pitt.edu}
\author{Hui Li}
\address{Hui Li, Penn State University, 2030 Mary Ellen Lane, State College, PA 16803, USA}
\email{huili@math.umn.edu}

\begin{abstract}
We prove convergence of critical points $u^h$ of the nonlinear
elastic energies $E^h$ of thin incompressible plates
$\Omega^h=\Omega \times (-h/2, h/2)$, which satisfy the von
K\'arm\'an scaling: $E^h(u^h)\leq Ch^4$, to critical points of the
appropriate limiting (incompressible von K\'arm\'an) functional.
\end{abstract}

\maketitle

\section{Introduction and the main result}
In this paper we prove convergence of critical points of the
nonlinear elastic energies on thin incompressible plates in the von
K\'arm\'an scaling regime, to critical points of the appropriate
limiting (incompressible von K\'arm\'an) functional.

\subsection{Elastic energy of thin incompressible plates}
Let $\Omega\subset \mathbb{R}^2$ be an open, bounded, simply
connected domain. For $h>0$, define $\Omega^h$ to be the $3$d plate
with the midplate $\Omega$ and thickness $h$:
\begin{equation*}\label{Omegah}
\Omega^h = \left\{x=(x', x_3);~ x'\in \Omega, ~x_3 \in
\left(-\frac{h}{2},
    \frac{h}{2}\right)\right\}.
\end{equation*}
The elastic energy of a deformation $u^h\in W^{1.2}(\Omega^h,
\mathbb{R}^3)$ of the homogeneous plate $\Omega^h$, scaled by its
unit thickness,  is given by:
\begin{equation}\label{ElasticE}
I^h(u^h) = \frac{1}{h} \int_{\Omega^h} W_{in}(\nabla u^h)~\mbox{d}x,
\end{equation}
while  the total energy, relative to the
 external force with the density $f^h \in L^2(\Omega^h,
\mathbb R^3)$, is:
\begin{equation}\label{TotalE}
J^h(u^h) = \frac{1}{h}\int_{\Omega^h}W_{in}(\nabla u^h)~ \mbox{d}x -
\frac{1}{h}\int_{\Omega^h}f^h \cdot u^h~ \mbox{d}x.
\end{equation}
The elastic energy density $W_{in}: \mathbb{R}^{3\times 3}
\rightarrow [0,\infty]$ in (\ref{ElasticE}) is assumed to be
infinite at compressible deformations:
$$ W_{in}(F) = \left\{
\begin{array}{ll}
W(F) & \mbox{ if}~~\det F = 1,\\
+\infty &\mbox{ otherwise}.
\end{array} \right. $$
The effective density $W: \mathbb R^{3 \times 3} \rightarrow
[0,\infty)$ above, which acts when $\det F= 1$, is required to
satisfy the following conditions:
\begin{itemize}
\item [(i)] {\em (frame invariance)} $W(RF) = W(F)$, for each proper
  rotation $R \in SO(3)$, and each $F \in \mathbb R^{3 \times 3}$.
\item [(ii)] {\em (normalisation)} $W(F) = 0$ for all  $F\in SO(3)$.
\item [(iii)] {\em (non-interpenetration)} $W(F) = + \infty$ if $\det
  F \leq 0$, and $W(F) \to + \infty$ as $\det F \to 0+$.
\item [(iv)] {\em (bound from below)} $W(F) \geq c ~\mbox{dist}^2(F,
  SO(3))$ with  a constant $c>0$  independent of $F$.
\item [(v)] {\em (bound from above)}
There exists  a constant $C>0$ such that for each $F$ with $\det
F>0$, i.e. for each $F\in \mathbb R_+^{3
  \times 3}$ there holds:
\begin{equation}\label{growth}
|DW(F)F^T| \leq C (W(F) + 1).
\end{equation}
\item [(vi)] {\em (regularity)} $W$ is of class $\mathcal C^1$ on $\mathbb R_+^{3 \times 3}$.
\item [(vii)] {\em (local regularity)} $W$ is of class $\mathcal C^2$
  in a small neighborhood of $SO(3)$.
\end{itemize}
The growth conditions in (iv) and (v) will be crucial in the present
analysis. Condition (iv) has been introduced in the context of
\cite{FJM} and it allows to use the nonlinear version of Korn's
inequality \cite{FJMgeo}, ultimately serving to control the local
deviations of the deformation $u^h$ from rigid motions, by the
elastic energy $I^h(u^h)$. Condition (v) has been introduced in
\cite{Ball1}  (see also \cite{Ball2}) in the context of inner
variations, in order to control the related strain in terms of the
energy. Both conditions are compatible with other requirements
above. Indeed, examples of  $W$ satisfying (i) -- (vii) are:
\begin{equation*}
\begin{split}
W_1(F) & = |(F^TF)^{1/2} - \mbox{Id}|^2 + |\log \det F|^q, \\
W_2(F) & =  |(F^TF)^{1/2} - \mbox{Id}|^2 + \left|\frac{1}{\det F} -
1\right|^q \mbox{ for } \det F>0,
\end{split}
\end{equation*}
where $q>1$ and $W$ equals $+\infty$ if $\det F\leq 0$ \cite{MS}.

\subsection{Notation}
Given a matrix $F\in \mathbb{R}^{n\times n}$, we denote its trace by
$\mbox{Tr}~ F$ and its transpose by $F^T$. The symmetric part of $F$
is given by $\mbox{sym}~ F = \frac{1}{2} (F + F^T)$. The cofactor of
$F$ is the matrix: $\mbox{cof}~ F$, where $[\mbox{cof } F]_{ij} =
(-1)^{i+j} \det \hat F_{ij}$ and each $\hat
F_{ij}\in\mathbb{R}^{(n-1)\times (n-1)}$ is obtained from $F$ by
deleting its $i$th row and $j$th column. The identity matrix is
denoted by $\mbox{Id}_n$.

In what follows, we shall use the matrix norm $|F|=
(\mbox{Tr}(F^TF))^{1/2}$, which is induced by the inner product:
$F_1:F_2 = \mbox{Tr}(F_1^TF_2)$. To avoid notational confusion, we
will often write $\langle F_1:F_2\rangle$ instead of $F_1:F_2$. In
general, $3\times 3$ matrices will be denoted by $F$ and $2\times 2$
matrices will be denoted by $F''$. Unless noted otherwise, $F''$ is
the principal $2\times 2$ minor of $F$.

Finally, by $\mathcal{C}^k_b(\mathbb{R}^n, \mathbb{R}^s)$ we denote
the space of continuous functions whose derivatives up to the order
$k$ are continuous and bounded in $\mathbb{R}^n$.

\subsection{The limiting energy}
The following $2$d energy functional has been rigorously derived in
\cite{LC} as the $\Gamma$-limit of the scaled incompressible
energies $h^{-4}I^h$ in (\ref{ElasticE}), when $h\to 0$:
\begin{equation}\label{VK}
\mathcal I(w, v) = \frac{1}{2}\int_{\Omega} \mathcal
Q_2^{in}\left(\mbox{sym}\nabla u + \frac{1}{2} \nabla v \otimes
  \nabla v\right)~\mbox{d}x + \frac{1}{24}\int_{\Omega}\mathcal
Q_2^{in}\left(\nabla^2 v\right)~\mbox{d}x,
\end{equation}
acting on couples $w \in W^{1,2}(\Omega,\mathbb{R}^2), v\in
W^{2,2}(\Omega,\mathbb{R})$. The fields $(w,v)$ may be identified as
the in-plane and the out-of-plane displacements, respectively.
Roughly speaking, any minimizing sequence of $h^{-4}J^h$, where
$f^h(x)  \approx h^3 f(x') e_3$ and $ \int_\Omega f = 0$,  will have
the structure:
$$u^h_{|\Omega} \approx (\bar R)^T\left(\mbox{id} + hv e_3 + h^2
  w\right) - c^h$$
asymptotically as $h\to 0$, with $(w,v)$ as above and $\bar R\in
SO(3)$ maximizing  $\int_\Omega f(x') e_3\cdot Rx' ~\mbox{d}x'$
among all rotations $R$, while  $c^h \in \mathbb R^3$ are constant
translation vectors.  Moreover, $(w,v,\bar R)$ minimize the
following total limiting energy:
$$\mathcal{J}(w, v,\bar R) = \mathcal{I}(w, v) - \bar R_{33}\int_\Omega fv.$$
A precise formulation of the statements above can be found in
\cite{LeMoPa}.

The energy in (\ref{VK}) is the incompressible version of the von
K\'arm\'an functional, which has been derived (for compressible
case, i.e. without the assumption that $\det\nabla u^h = 1$) by
means of $\Gamma$-convergence in \cite{FJM}. The quadratic forms
$\mathcal{Q}_2^{in}$ differ from the standard $\mathcal{Q}_2$ in
\cite{FJM} in as much as minimization in (\ref{Q2}) below is taken
over the out-of-plane stretches which preserve the incompressibility
constraint. Namely, $\mathcal Q_2^{in} $ in (\ref{VK}) are given as:
\begin{equation}\label{Q2}
\begin{split}
&\forall F''\in \mathbb{R}^{2\times 2} \quad \mathcal Q_2^{in}(F'')
= \min_{d \in \mathbb R^3} \Big\{ Q_3(F'' + d \otimes e_3 + e_3
\otimes d);~ \mbox{Tr}(F'' + d \otimes e_3 + e_3
\otimes d) = 0\Big\}, \\
&\forall F\in \mathbb{R}^{3\times 3} \quad \mathcal Q_3(F) = D^2
W(\mbox{Id})(F, F).
\end{split}
\end{equation}
Both forms $\mathcal Q$ above are positive semidefinite, and
strictly positive definite on symmetric matrices. We also introduce
the linear operators $\mathcal L_2^{in}:\mathbb{R}^{2\times
2}\rightarrow \mathbb{R}^{2\times 2}$ and $\mathcal L_3:
\mathbb{R}^{3\times 3}\rightarrow \mathbb{R}^{3\times 3}$ such that:
\begin{equation}\label{Ls}
\begin{split}
& \forall F''\in \mathbb{R}^{2\times 2} \quad
\langle\mathcal{L}_2^{in}(F'') : F'' \rangle = \mathcal{Q}_2^{in}(F''),\\
&\forall F\in \mathbb{R}^{3\times 3} \quad  \langle\mathcal{L}_3(F)
: F\rangle = \mathcal{Q}_3(F).
\end{split}
\end{equation}
Note that symmetric operators $\mathcal{L}$ are uniquely given by:
$\langle\mathcal{L}(F_1):F_2\rangle =
\frac{1}{4}\left(\mathcal{Q}(F_1+F_2)  - \mathcal{Q}(F_1-F_2)\right)$.\\

\subsection{Critical points and the incompressible inner variations}

Following \cite{Ball2}, we now define the critical points $u^h$ of
the functionals $J^h$ in (\ref{TotalE}) with respect to inner
variations, that is requesting that the derivative of $J^h$ at an
incompressible equilibrium $u^h$ be zero:
$$\frac{\mathrm{d}}{\mathrm{d}\epsilon} _{|\epsilon=0}
J^h(u^h_\epsilon) = 0,$$ along all curves $\epsilon\mapsto
u^h_\epsilon$ of incompressible  deformations of $\Omega^h$ having
the form: $u^h_\epsilon (x) = \Phi(\epsilon, u^h(x))$, with  $u^h_0
= u^h$ at $\epsilon = 0$. This requirement is translated into the
following condition:
\begin{equation}\label{EqEqn}
\int_{\Omega^h}  \left\langle DW(\nabla u^h)(\nabla u^h)^T:\nabla
\phi(u^h(x))\right\rangle~\mbox{d}x = \int_{\Omega^h}f^h \cdot
\phi(u^h)~\mbox{d}x, \quad \forall \phi \in \mathcal{C}_b^1(\mathbb
R^3,\mathbb R^3) \mbox{ with } \mbox{div} ~\phi = 0.
\end{equation}
We refer to section \ref{sec_incomp} for the derivation and
discussion of (\ref{EqEqn}). Let us only note now that the
incompressible inner variations:
\begin{equation*}\label{expansionUe}
u^h_{\epsilon}(x) = \Phi(\epsilon, u^h(x)) = u^h(x) + \epsilon
\phi(u^h(x)) + \mathcal O(\epsilon^2).
\end{equation*}
replace the classical variations $u^h_{\epsilon}(x) = u^h(x) +
\epsilon w^h(x)$ used in definition of minimizers of $J^h$, and also
they replace the inner variations $u^h_{\epsilon}(x) = u^h(x) +
\epsilon \phi(u^h(x))$ considered in \cite{Ball2} and \cite{MS} in
the compressible case.

\subsection{The main result}
The following is our main result:

\begin{theorem}\label{thmain}
For each $h<<1$, let $u^h\in W^{1,2}(\Omega^h,\mathbb{R}^3)$ be a
critical point of $J^h$, i.e. it satisfies (\ref{EqEqn}) subject to
the external forces $f^h(x) = h^3 f(x')e_3$. Assume that:
\begin{equation}\label{en_as}
I^h(u^h) \leq Ch^4,
\end{equation}
for a constant $C>0$ independent of $h$. Then there exists a
sequence of proper rotations $\bar R^h \in SO(3)$, and translations
$c^h \in \mathbb R^3$, such that for the renormalized deformations:
\begin{equation}\label{yh}
y^h(x', x_3) = (\bar R^h)^T u^h(x', hx_3) - c^h \in
W^{1,2}(\Omega^1,\mathbb{R}^3),
\end{equation}
the following convergences hold, up to a subsequence in $h$, as
$h\to 0$:
\begin{itemize}
\item [(i)] $\bar R^h \to \bar R = [\bar R_{ij}]_{i,j:1..3}\in SO(3)$.
\item [(ii)] $y^h \to x'$ in $W^{1, 2}(\Omega^1)$.
\item [(iii)] For the scaled out-of-plane displacements:
\begin{equation}\label{Sc_Nor_Dis}
v^h(x') = \frac{1}{h}\int_{-1/2}^{1/2} y_3^h(x', x_3)
~\mathrm{d}x_3,
\end{equation}
there exists $v \in W^{2, 2}(\Omega, \mathbb{R})$ such that $v^h \to
v$ strongly in $W^{1, 2}(\Omega)$.
\item [(iv)] For the scaled in-plane displacements:
\begin{equation}\label{Sc_Tan_Dis}
w^h(x') = \frac{1}{h^2}\int_{-1/2}^{1/2}\left((y^h)'(x', x_3) -
x'\right)~\mathrm{d}x_3
\end{equation}
there exists $w \in W^{1, 2}(\Omega, \mathbb R^2)$ such that
$w^h\rightharpoonup w $ weakly in $W^{1, 2}(\Omega, \mathbb R^2)$.
\item [(v)] The limiting displacements $(w, v)$ solve the following
  Euler-Lagrange equations of the functional (\ref{VK}), expressed in
  the variational form:
\begin{equation}\label{EL1}
\int_{\Omega}\left\langle\mathcal L_2^{in}\left(\mathrm{sym} \nabla
w +
  \frac{1}{2}\nabla v \otimes \nabla v\right) : \nabla \tilde w \right\rangle
~\mathrm{d}x' = 0
\end{equation}
\begin{equation}\label{EL2}
\begin{split}
&\int_{\Omega}\left\langle\mathcal L_2^{in}\left(\mathrm{sym}\nabla
w +
    \frac{1}{2} \nabla v \otimes \nabla v\right) : (\nabla v \otimes
  \nabla \tilde v)\right\rangle ~\mathrm{d}x' \\
&\qquad\qquad\qquad \qquad\qquad\qquad
  + \frac{1}{12} \int_{\Omega}\left\langle\mathcal L_2^{in}(\nabla^2 v) : \nabla^2
  \tilde v\right\rangle ~\mathrm{d}x' = \bar R_{33} \int_{\Omega}
f\tilde v ~\mathrm{d}x',
\end{split}
\end{equation}
for every $\tilde w \in W^{1,2}(\Omega, \mathbb R^2)$ and every
$\tilde v \in W^{2,2}(\Omega,\mathbb{R})$.
\end{itemize}
\end{theorem}
We note that (\ref{en_as}) are automatically satisfied by any
minimizing sequence of $u^h$ of the total energy $J^h$, under the
assumption that $f^h(x) = h^3 f(x')e_3$ \cite{FJM}.  Also,
(\ref{EqEqn}) holds for every minimum of $J^h$ (see Theorem
\ref{deriv}), and the assertions (i) - (v) are then a direct
consequence \cite{LC} of the fact that $\frac{1}{h^4}J^h$
$\Gamma$-converges to $\mathcal{J}$. In general,
$\Gamma$-convergence does not assure that a limit of a sequence of
equlibria is an equilibrium of the $\Gamma$-limit. In the present
situation, this turns out to be the case.

\subsection{Relation to other works}
Our work is largely inspired by \cite{MS} and \cite{LC}. To put it
in a larger perspective, recall that one of the fundamental
questions in the mathematical theory of elasticity has been to
rigorously justify various 2d plate models present in the
engineering literature, in relation to the three-dimensional theory.
This goal has been largely accomplished in \cite{FJM}, where a
hierarchy of limiting 2d energies has been derived; the distinct
theories are differentiated by their validity in the corresponding
scaling regimes $h^\beta,$ $\beta\geq 2$, i.e. in presence of
assumption (\ref{en_as}) where $h^4$ is replaced by $h^\beta$.

Under the additional incompressibility constraint, the works
\cite{5,
  6} proved compactness properties and the
$\Gamma$-convergence of the functionals $\frac{1}{h^\beta} I^h$ as
in (\ref{ElasticE}), for the so-called Kirchhoff scaling $\beta =
2$, while \cite{LC} treated the case $\beta=4$ including as well a
more complex case of shells when the midsurface $\Omega$ is a
generic 2d hypersurface in $\mathbb{R}^3$. In view of the
fundamental property of $\Gamma$-convergence, it follows that the
global almost-minimizers of the energies (\ref{TotalE}) converge to
the minimizers of the limiting energy (given by (\ref{VK}) in the
von K\'arm\'an regime).

Regarding convergence of stationary points for thin plates,  the
first result has been obtained in \cite{MP} under the von K\'arm\'an
scaling  $\beta = 4$ (see also \cite{L1} for an extension to thin
shells). These results relied on the crucial assumption that the
elastic energy density $W$ is differentiable everywhere and its
derivative satisfies a linear growth condition: $|DW(F)|\leq C(|F| +
1)$. This assumption is contradictory with the physically expected
non-interpenetration condition, and subsequently it has been removed
in \cite{MS} and  exchanged with Ball's condition (\ref{growth}),
while  the equilibrium equations have been rephrased in terms of the
inner variations. In the present paper we follow the same approach;
indeed the concept of inner variations comes up naturally in the
context of incompressible elasticity.

\medskip

To conclude, we now comment on the isotropic case. For an isotropic
energy density $W$ with the Lam\'e constants $\lambda$  and $\mu$,
the Euler-Lagrange equations (\ref{EL1}) -- (\ref{EL2}) of
(\ref{VK}) are:
\begin{equation}\label{EL}
\frac{\mu}{3} \Delta^2 v = [v,\Phi], \qquad \Delta^2\Phi =
-\frac{3\mu}{2} [v,v],
\end{equation}
where $v$ is the out-of-plane displacement, while the in-plane
displacement $w$ can be recovered through the Airy stress potential
$\Phi$, by means of:
$$\mbox{cof}\nabla^2 \Phi =
2\mu\Big[\mbox{sym} \nabla w + \frac{1}{2}\nabla v\otimes \nabla v +
\Big(\mbox{div} w + \frac{1}{2} |\nabla v|^2\Big)\mbox{Id}\Big].$$
The Airy's bracket $[\cdot, \cdot]$ is defined as: $[v,\Phi] =
\nabla^2v : (\mbox{cof}\nabla^2\Phi)$. As expected, the system
(\ref{EL}) can be now obtained as the incompressible limit, i.e.
when passing with the Poisson ratio $\nu\to \frac{1}{2}$, of the
classical (compressible) von K\'arm\'an system:
\begin{equation*}\label{EL_compre}
B\Delta^2 v = [v,\Phi], \qquad \Delta^2\Phi = -\frac{S}{2} [v,v],
\end{equation*}
where $S=2\mu(1+\nu)$ is Young's modulus,
$\nu=\frac{\lambda}{2(\mu+\lambda)}$ is the Poisson ratio, and
$B=\frac{S}{12(1-\nu^2)}$ is bending stiffness. By the change of
variable $\Phi = 2\mu\Phi_1$ one can eliminate the parameter $\mu$
entirely and write (\ref{EL}) in its equivalent form:
\begin{equation*}
\Delta^2 v = 6[v,\Phi_{1}], \qquad \Delta^2\Phi_1 = -\frac{3}{4}
[v,v].
\end{equation*}

\bigskip

\noindent{\bf Acknowledgments.} M.L. was partially supported by the
NSF Career grant DMS-0846996 and by the Polish MN grant N N201
547438.

\section{Incompressible inner variations and critical points}\label{sec_incomp}

Following \cite{Ball2}, we want to define the critical points $u^h$
of the functionals $J^h$ in (\ref{TotalE}) by taking inner
variations. That is, we request that the derivative of $J^h$ at an
incompressible equilibrium $u^h$ be zero along all curves
$\epsilon\mapsto u^h_\epsilon$ of incompressible  deformations of
$\Omega^h$ having the form: $u^h_\epsilon (x) = \Phi(\epsilon,
u^h(x))$, with  $u^h_0 = u^h$ at $\epsilon = 0$. This requirement
imposes the following conditions on the flow
$\Phi:[0,\epsilon_0)\times\mathbb{R}^3\rightarrow \mathbb{R}^3$:
\begin{equation}\label{16}
\begin{split}
& \forall \epsilon \quad \Phi(\epsilon, \cdot) \mbox{ is
  incompressible, i.e } \quad \forall y\in\mathbb{R}^3\quad \det\nabla\Phi(\epsilon, y) = 1,\\
& \forall y\in\mathbb{R}^3\quad \Phi(0,y) = y.
\end{split}
\end{equation}
Assuming sufficient smoothness of $\Phi$, the above immediately
implies:
\begin{equation*}
\begin{split}
0 & = \frac{\mathrm{d}}{\mathrm{d}\epsilon} \det \nabla \Phi(0, y) =
\left\langle\mbox{cof}\nabla\Phi(0,y) :
\frac{\mathrm{d}}{\mathrm{d}\epsilon} \nabla\Phi(0,y) \right\rangle
= \left\langle\mbox{Id} :
\frac{\mathrm{d}}{\mathrm{d}\epsilon} \nabla\Phi(0,y) \right\rangle\\
& = \mbox{Tr} \left( \frac{\mathrm{d}}{\mathrm{d}\epsilon}
  \nabla\Phi(0,y)\right)
= \mbox{div} \left(\frac{\mathrm{d}}{\mathrm{d}\epsilon}
  \Phi(0,y)\right) =: \mbox{div}~\phi(y).
\end{split}
\end{equation*}
On the other hand, any divergence-free vector field $\phi$ generates
a path of incompressible deformations. We recall this standard fact
below, for the sake of completeness.

\begin{lemma}\label{Inner_vari}
Let $\phi\in \mathcal{C}^1_b(\mathbb{R}^n, \mathbb{R}^n)$ such that
$\mathrm{div}~ \phi = 0$. Consider the ODE:
\begin{equation}\label{ODE}
\left\{
\begin{array}{ll}
u'(\epsilon) = \phi(u(\epsilon)),\\
u(0) = y.
\end{array} \right.
\end{equation}
and denote its flow by $\Phi(\epsilon, y) = u(\epsilon)$ solving
(\ref{ODE}). Then $\Phi$ satisfies (\ref{16}).
\end{lemma}
\begin{proof}
Let $\epsilon, \delta > 0$ and note that: $ \Phi(\epsilon+\delta, y)
= \Phi(\delta, \Phi(\epsilon, y)) = \Phi(\delta, y_1)$ where we put
$y_1 = \Phi(\epsilon, y)$. Hence, denoting the spacial gradient by
$\nabla$, we obtain:
\begin{equation*}
\det \nabla \Phi(\epsilon+\delta, y) = \det \nabla \Phi(\delta, y_1)
\det \nabla \Phi(\epsilon, y),
\end{equation*}
Consequently:
\begin{equation}\label{Mainsth}
\begin{split}
\frac{\mbox{d}}{\mbox{d}\epsilon}\left(\det \nabla
  \Phi(\epsilon+\delta, y)\right) & =
\frac{\mbox{d}}{\mbox{d}\delta}\left(\det \nabla
\Phi(\epsilon+\delta, y)\right)
 = \frac{\mbox{d}}{\mbox{d}\delta}\left(\det \nabla\Phi(\delta, y_1)\right)\left(\det
  \nabla\Phi(\epsilon, y)\right) \\ & =\left\langle\mbox{cof}~ \nabla \Phi(\delta,
  y_1) : \frac{\mbox{d}}{\mbox{d}\delta} \nabla \Phi
  (\delta,y_1)\right\rangle \det \nabla \Phi(\epsilon, y).
\end{split}
\end{equation}
Above, we used the formula for the derivative of the determinant of
a matrix function $A(t)$, namely: $(\det A(t))' = \mbox{cof} A(t):
A(t)'. $ For $\delta = 0$, (\ref{Mainsth}) implies:
$$ \frac{\mbox{d}}{\mbox{d}\epsilon}\left(\det \nabla
  \Phi(\epsilon, y)\right) = \langle \mbox{cof} \nabla \Phi(0, y_1) : \nabla \phi(y_1)\rangle =
\langle \mbox{Id} : \nabla \phi(y_1)\rangle = \mbox{Tr} \nabla \phi
= \mbox{div}~ \phi = 0. $$ But  $\det \nabla \Phi(0, y) = \det
\mbox{Id}_n = 1$, which achieves the claim.
\end{proof}

We are now ready to derive the equilibrium equations (\ref{EqEqn}).
The result is essentially similar to Theorem 2.4 \cite{Ball2}, which
dealt with the compressible inner variations  $u^h_{\epsilon} =
u^h(x) + \epsilon \phi\circ u^h$ of a deformation $u^h$ with clamped
boundary conditions. The growth condition (\ref{growth}) will be
crucial in passing to the limit in the nonlinear term in $J^h$, to
which end we are going to use the following Lemma from \cite{Ball2}:

\begin{lemma} \label{ball2} (Lemma 2.5 (i) \cite{Ball2})
Assume that $W$ satisfies (\ref{growth}). Then there exists $\gamma
> 0$ such that if $A\in\mathbb{R}^{3\times 3}_+$ and
$|A-\mathrm{Id}|<\gamma$, then:
$$|DW(AF)F^T|\leq 3C (W(F) +1) \qquad \forall F\in \mathbb{R}^{3\times 3}_+,$$
where $C$ is the constant in condition (\ref{growth}).
\end{lemma}

\begin{theorem}\label{deriv}
Let $\phi\in \mathcal{C}^1_b(\mathbb{R}^3, \mathbb{R}^3)$ be such
that $\mathrm{div}~ \phi = 0$. Given a deformation $u^h\in
W^{1,2}(\Omega^h,\mathbb{R}^3)$ with $\det\nabla u^h = 1$, and such
that $\int_{\Omega^h} W(\nabla u^h)~\mathrm{d}x <+\infty$, define
$u_\epsilon^h(x) = \Phi(\epsilon, u^h(x))$. Then:
$$\frac{\mathrm{d}}{\mathrm{d}\epsilon} _{|\epsilon=0}
J^h(u^h_\epsilon) = 0$$ is equivalent to:
\begin{equation*}
\int_{\Omega^h}  \left\langle DW(\nabla u^h)(\nabla u^h)^T:\nabla
\phi(u^h(x))\right\rangle~\mathrm{d}x = \int_{\Omega^h}f^h \cdot
\phi(u^h)~\mathrm{d}x.
\end{equation*}
\end{theorem}
\begin{proof}
For the notational convenience, in what follows we drop the index
$h$ and write $U$ instead of $\Omega^h$, which stands now for a
fixed open bounded domain in $\mathbb{R}^3$. It is easy to notice
that:
\begin{equation}\label{uniform}
\lim_{\epsilon\to 0}\frac{1}{\epsilon}(\Phi(\epsilon, y) - y) =
\phi(y) \quad \mbox{ uniformly in } \mathbb{R}^3.
\end{equation}
It directly implies that:
$$\lim_{\epsilon\to 0}\frac{1}{\epsilon} \int_U f \cdot
(\Phi(\epsilon, u(x)) - u(x))~\mbox{d}x = \int_U f\cdot \phi(u(x))
~\mbox{d}x. $$ To treat the nonlinear term, consider:
\begin{equation}\label{non}
\begin{split}
\frac{1}{\epsilon}&\int_U \big( W(\nabla u_\epsilon) - W(\nabla
u)\big)~\mbox{d}x - \int_U \big \langle DW(\nabla u) (\nabla u)^T :
\nabla\phi(u)\big\rangle ~\mbox{d}x \\ & = \int_U \fint_0^\epsilon
\Big\langle DW\big(\nabla\Phi (s,u)\nabla u\big) (\nabla u)^T:
\nabla \phi(\Phi(s, u))\Big\rangle  - \Big \langle DW(\nabla u)
(\nabla u)^T : \nabla\phi(u)\Big\rangle ~\mbox{d}s ~\mbox{d}x.
\end{split}
\end{equation}
Since  the integrand below converges to $0$ pointwise by
(\ref{uniform}), and it is bounded by the function
$2\|\nabla\phi\|_{L^\infty} |DW(\nabla u)(\nabla u)^T|$ which is
integrable in view of (\ref{growth}), we obtain:
$$ \lim_{\epsilon\to 0}\int_U
\Big\langle DW(\nabla u) (\nabla u)^T: \fint_0^\epsilon \nabla
\phi(\Phi(s, u)) - \nabla\phi(u) ~\mbox{d}s \Big\rangle ~\mbox{d}x =
0,$$ by the dominated convergence theorem. Similarly:
$$ \lim_{\epsilon\to 0}\int_U  \fint_0^\epsilon
\Big\langle \Big(DW(\nabla \Phi(s, u)\nabla u) - DW(\nabla u)\Big)
(\nabla u)^T: \nabla \phi(\Phi(s, u)) \Big\rangle ~\mbox{d}s
~\mbox{d}x = 0,$$ where the pointwise convergence follows by the
formula (\ref{uniform}), its counterpart for $\nabla \Phi$, and the
continuity of $DW$ on $\mathbb{R}^{3\times 3}_+$. The integrands,
for small $\epsilon$,  are dominated by the $L^1(U)$ function
$4C\|\nabla\phi\|_{L^\infty}(W(\nabla u) +1)$ in view of Lemma
\ref{ball2} and the growth condition (\ref{growth}).

Therefore, the left hand side in (\ref{non}) converges to $0$ as
well. This completes the proof.
\end{proof}

\section{The equilibrium equation (\ref{EqEqn})}
In this section, we review several facts from \cite{FJM} and
\cite{MS}, to set the stage for a proof of Theorem \ref{thmain} and
to rewrite the equation (\ref{EqEqn}) using the change of variables
(\ref{yh}).

The first crucial step in the dimension reduction argument of
\cite{FJM} is finding the appropriate approximations of  the
deformations gradients $u^h$. Under the sole assumption:
\begin{equation}\label{en_as2}
\frac{1}{h} \int_{\Omega^h} W(\nabla u^h)~\mbox{d}x \leq Ch^4,
\end{equation}
an application of a nonlinear verion of Korn's inequality
\cite{FJMgeo}, yields existence of rotation fields $R^h\in
W^{1,2}(\Omega, \mathbb{R}^{3\times 3})$ with $R^h(x) \in SO(3)$
a.e. in $\Omega$, so that:
\begin{equation}\label{uno}
\|\nabla u^h(x', hx_3) - R^h\|_{L^2(\Omega^1)} \leq Ch^2 \qquad
\mbox{and} \qquad  \|\nabla R^h\|_{L^2(\Omega)} \leq Ch.
\end{equation}
Recall that $\Omega^1=\Omega\times (-\frac{1}{2}, \frac{1}{2})$ is
the common domain of the rescaled deformations $y^h(x', x_3) = (\bar
R^h)^T u^h(x', hx_3) - c^h$, and the typical point in $\Omega^1$ is
denoted by $x=(x', x_3)$. Then, the detailed analysis in \cite{FJM}
shows that convergences in (i) -- (iv) of Theorem \ref{thmain} hold,
as a consequence of (\ref{en_as}) implying (\ref{en_as2}). The
constant rotations $\bar R^h\in SO(3)$ are given by:
$$\bar R^h = \mathbb{P}_{SO(3)} \left(\fint_{\Omega^h} \nabla u^h~\mbox{d}x\right),$$
where the orthogonal projection $ \mathbb{P}_{SO(3)} $ onto $SO(3)$
above is well defined; see also \cite{LMP} for detailed
calculations. Further, there holds:
\begin{equation}\label{due}
\|R^h - \bar R^h\|_{L^2(\Omega)} \leq Ch \qquad \mbox{and} \qquad
\lim_{h\to 0} (\bar R^h)^T R^h = \mbox{Id} \quad \mbox{in }
W^{1,2}(\Omega, \mathbb{R}^{3\times 3}),
\end{equation}
and upon defining the matrix fields $A^h \in W^{1, 2}(\Omega,
\mathbb R^{3 \times 3})$:
\begin{equation}\label{Ah}
A^h (x') = \frac{1}{h}\left((\bar R^h)^T R^h(x') - \mbox{Id}\right),
\end{equation}
it also follows that:
\begin{equation}\label{CAh}
A^h \rightharpoonup A = \left[\begin{array}{c|c} 0 &
\begin{minipage}{0.8cm}\[\vspace{3mm} -\nabla v\]\end{minipage} \\
\hline \nabla v & \begin{minipage}{0.8cm}\[ \vspace{3mm}
0\]\end{minipage}
\end{array} \right]
 \qquad \mbox{weakly in } W^{1, 2}(\Omega, \mathbb R^{3 \times 3}).
\end{equation}
The same convergence holds strongly in $L^{q}(\Omega, \mathbb R^{3
  \times 3})$ for each $q \geq 1$.

\begin{lemma}
We have:
\begin{equation}\label{convi}
\lim_{h\to 0} y^h = (x',0) \quad \mbox{ and } \quad \lim_{h\to 0}
\frac{y^h_3}{h} = x_3 + v(x')  \qquad \mbox{in } ~W^{1,2}(\Omega^1).
\end{equation}
Consequently, for every $\omega_h>0$ and $p\in [1, 5]$:
\begin{equation}\label{convi2}
\left|\left\{x\in\Omega^1;~ \frac{|y^h_3(x)|}{h}\geq \omega_h
\right\}\right| \leq \frac{C}{\omega_h^2}\quad \mbox{ and } \quad
\int_{\left\{x\in\Omega^1;~ \frac{|y^h_3(x)|}{h}\geq \omega_h
\right\}} \left|\frac{y^h_3(x)}{h}\right|^p~\mathrm{d}x \leq
\frac{C}{\omega_h^{\frac{2}{p+1}}}.
\end{equation}
\end{lemma}
\begin{proof}
By (\ref{uno}), (\ref{due}), and applying the Poincar\'e-Wirtinger
inequality on segments $\{x'\}\times (-\frac{1}{2}, \frac{1}{2})$,
we see that:
\begin{equation*}
\begin{split}
\left\|\frac{y_3^h}{h} - x_3 - v^h(x')\right\|_{L^2(\Omega^1)} &
\leq C \Big\|\frac{\partial_3 y_3^h}{h} - 1\Big\|_{L^2(\Omega^1)} =
C \Big\|[(\bar R^h)^T \nabla u^h(x', hx_3)]_{33} - 1\Big\|_{L^2(\Omega^1)} \\
& \leq C \|(\bar R^h)^T\nabla u^h(x', hx_3) -
\mbox{Id}\|_{L^2(\Omega^1)} \\ & \leq C \|\nabla u^h(x', hx_3) -
R^h\|_{L^2(\Omega^1)} + C \|R^h - \bar R^h\|_{L^2(\Omega^1)} \leq
Ch.
\end{split}
\end{equation*}
Together with (\ref{Sc_Nor_Dis}), the above inequality implies the
second assertion in (\ref{convi}). The first assertion follows then
directly in view of (\ref{Sc_Tan_Dis}).

To prove (\ref{convi2}), note that for every $p\in [1,5]$:
\begin{equation}\label{eins}
\begin{split}
\int_{\left\{x\in\Omega^1;~ \frac{|y^h_3(x)|}{h}\geq \omega_h
\right\}} \left|\frac{y^h_3(x)}{h}\right|^p~\mathrm{d}x & \leq
\left\|\frac{y^h_3}{h}\right\|^p_{L^{p+1}}
\left|\left\{x\in\Omega^1;~ \frac{|y^h_3(x)|}{h}\geq \omega_h
  \right\}\right|^{\frac{1}{p+1}} \\
& \leq C \left|\left\{x\in\Omega^1;~ \frac{|y^h_3(x)|}{h}\geq
\omega_h
  \right\}\right|^{\frac{1}{p+1}},
\end{split}
\end{equation}
by the H\"older inequality and the Sobolev embedding
$W^{1,2}(\Omega^1) \hookrightarrow L^6(\Omega^1)$ combined with
(\ref{convi}). When $p=1$, it implies:
$$ \left|\left\{x\in\Omega;~ \frac{|y^h_3(x)|}{h}\geq \omega_h  \right\}\right|
\leq \frac{1}{\omega_h} \int_{\left\{x\in\Omega;~
\frac{|y^h_3(x)|}{h}\geq \omega_h  \right\}}
\frac{|y^h_3(x)|}{h}~\mathrm{d}x \leq \frac{C}{\omega_h}
\left|\left\{x\in\Omega;~ \frac{|y^h_3(x)|}{h}\geq \omega_h
  \right\}\right|^{1/2}  $$
Hence, the first assertion in (\ref{convi2}) follows, as well as the
second one, in view of (\ref{eins}).
 \end{proof}

\bigskip

Define the  strain $G^h \in L^2(\Omega^1, \mathbb R^{3
  \times 3})$ and the scaled stress $E^h\in L^1(\Omega^1, \mathbb R^{3 \times 3})$ as:
\begin{equation*}\label{GhEh}
\begin{split}
& G^h(x', x_3) = \frac{1}{h^2} \left((R^h)^T \nabla u^h(x', hx_3) -
  \mbox{Id}\right),\\
& E^h(x', x_3) = \frac{1}{h^2} DW(\mbox{Id} + h^2 G^h(x',
x_3))(\mbox{Id} + h^2 G^h(x', x_3))^T.
\end{split}
\end{equation*}
We now gather the fundamental properties of $E^h$ and $G^h$ from
\cite{MS}, that will be used in the sequel.

\begin{lemma}\label{lemMS}  (Section 4, \cite{MS})
\begin{itemize}
\item[(i)] Up to a subsequence, $G^h\rightharpoonup G$ weakly in $L^2(\Omega^1, \mathbb
{R}^{3 \times 3})$, where $G$ is the limiting strain whose principal
$2\times 2$ minor $G''$ satisfies:
\begin{equation}\label{Lim_Gh}
\begin{split}
& G''(x', x_3) = G_0(x') - x_3 G_1(x'), \quad \mbox{ with: }\\
& \mathrm{sym}~ G_0 = \mathrm{sym}\nabla w + \frac{1}{2}\nabla
v\otimes \nabla v, \qquad G_1 = \nabla^2 v.
\end{split}
\end{equation}
\item[(ii)] Each $E^h(x)$ is symmetric, and there holds:
\begin{equation}\label{tre}
|E^h| \leq C \left(\frac{1}{h^2}W(\mathrm{Id} + h^2 G^h) +
|G^h|\right).
\end{equation}
\item[(iii)] Up to a subsequence, $E^h\rightharpoonup E$ weakly
in $L^1(\Omega^1, \mathbb{R}^{3\times 3})$, and $E=\mathcal{L}_3(G)
\in L^2(\Omega^1, \mathbb{R}^{3\times 3})$.
\item[(iv)] For a given, fixed $\gamma\in (0,2)$, define $B_h = \{x\in\Omega^1;
  ~ h^{2-\gamma}|G^h(x)|\leq 1\}$. Then:
\begin{equation}\label{quattro}
|\Omega^1\setminus B_h|\leq Ch^{2(2-\gamma)} \qquad \mbox{and}
\qquad \int_{\Omega^1\setminus B_h} |E^h| ~\mathrm{d}x \leq
Ch^{2-\gamma}.
\end{equation}
Moreover, calling  $\chi_h$  the
  characteristic function of $B_h$, we have:
\begin{equation}\label{cinque}
\chi_h E^h \rightharpoonup E \quad \mbox{weakly in } L^2(\Omega^1,
\mathbb{R}^{3\times 3}).
\end{equation}
\end{itemize}
\end{lemma}

The below more convenient form of the equilibrium condition will be
repeatedly used in the proof of Theorem \ref{thmain}.
\begin{lemma} Condition (\ref{EqEqn}) is equivalent to:
\begin{equation}\label{Eq2}
\begin{split}
&\int_{\Omega^1}\left\langle (\bar R^h)^T R^hE^h(x', x_3) (R^h)^T
\bar
  R^h : \nabla \phi(y^h(x', x_3))\right\rangle~\mathrm{d}x_3~\mathrm{d}x' \\
&\qquad\qquad\qquad\qquad\qquad =  h \int_{\Omega^1} \left\langle
f(x')e_3, \bar R^h
  \phi(y^h(x', x_3))\right\rangle ~\mathrm{d}x_3~\mathrm{d}x',
\end{split}
\end{equation}
for each $\phi \in \mathcal{C}_b^1(\mathbb R^3, \mathbb R^3)$ with
$\mathrm{div} \phi = 0$.
\end{lemma}
\begin{proof}
For a given divergence free $\phi \in \mathcal{C}^1_b(\mathbb R^3,
\mathbb R^3)$, consider:
\begin{equation*}\label{Ch_Va}
\psi(y) = \bar R^h \phi\left((\bar R^h)^T y - c^h\right),
\end{equation*}
which satisfies $\psi\in\mathcal{C}^1_b$ and $\mbox{div} \psi=0$,
and moreover:
$$ \nabla \psi\left(u^h(x', hx_3)\right) = \bar R^h \nabla
\phi\left(y^h(x', x_3)\right)(\bar R^h)^T. $$ Use now (\ref{EqEqn})
with the divergence-free test function $\psi$:
\begin{equation*}\label{Aft_Ch_V_Eq}
\begin{split}
&\int_{\Omega}\int_{-1/2}^{1/2} \left\langle DW\left(\nabla u^h(x',
  hx_3)\right)\left(\nabla u^h(x', hx_3)\right)^T : \bar R^h \nabla
\phi(y^h(x', x_3))(\bar R^h)^T \right\rangle~\mbox{d}x_3~\mbox{d}x' \\
&\hspace{6cm}= h^3\int_{\Omega}\int_{-1/2}^{1/2}  f(x')e_3 \cdot
\bar R^h \phi(y^h(x', x_3))~\mbox{d}x_3~\mbox{d}x'.
\end{split}
\end{equation*}
The formula (\ref{Eq2}) now follows directly, in view of:
\begin{equation*}\label{DWT}
\begin{split}
DW(\nabla u^h(x', hx_3))(\nabla u^h(x', hx_3))^T & = R^h
DW(\mbox{Id} +
h^2 G^h(x))(\mbox{Id} + h^2 G^h(x))^T (R^h)^T \\
& =  h^2 R^h E^h(x', x_3)(R^h)^T.
\end{split}
\end{equation*}
\end{proof}

\section{Identification of the operators in (\ref{EL1}) -- (\ref{EL2})}

\begin{lemma}\label{lemma4.1}
Let $G\in\mathbb{R}^{3\times 3}$ and a symmetric matrix
$E\in\mathbb{R}^{3\times 3}$ satisfy:
$$\mathcal{L}_3(G) = E, \quad \mathrm{Tr}~G = 0 \quad \mbox{and}
\quad E_{13} = E_{23} = 0.$$ Then:
\begin{equation}\label{L2form}
\mathcal{L}_2^{in}(G'') = E'' -  E_{33} \mathrm{Id}_2.
\end{equation}
\end{lemma}
\begin{proof}
Since $\mathcal{L}$ and $\mathcal{Q}$ depend only on the symmetric
parts of their arguments, we may without loss of generality assume
that $G$ is symmetric.

Firstly, by definitions in (\ref{Q2}), (\ref{Ls}), it follows that
for every $F''\in\mathbb{R}^{2\times 2}$ there is a unique
tangential minimizer $d= d(F'')\in\mathbb{R}^2$, in the sense that:
\begin{equation}\label{dmin}
\mathcal{Q}_2^{in}(F'') = \mathcal{Q}_3 (\left[\begin{array}{cc} F''
&
    d\\ d & - \mbox{Tr } F''\end{array}\right]) \quad
\mbox{ and } \quad \Big\langle\mathcal{L}_3(\left[\begin{array}{cc}
F'' &
    d\\ d & - \mbox{Tr } F''\end{array}\right]) :
\left[\begin{array}{cc} 0 &
    c\\ c & 0 \end{array}\right]\Big\rangle = 0 \quad \forall c\in\mathbb{R}^2.
\end{equation}
The second identity above is just the Euler-Lagrange equation for
the minimization in (\ref{Q2}). By convexity of this minimization
problem, it also follows that $d$ is linear:
\begin{equation}\label{lind}
d(F'' + G'') = d(F'') + d(G'')
\end{equation}

Observe now that:
\begin{equation*}\label{Q2form}
\begin{split}
\mathcal{Q}_2(G'') & = \mathcal{Q}_3 (\left[\begin{array}{cc} G'' &
    d(G'')\\ d(G'') & G_{33}\end{array}\right])
= \Big\langle\mathcal{L}_3 (\left[\begin{array}{cc} G'' &
    d(G'')\\ d(G'') & G_{33}\end{array}\right]) : \left[\begin{array}{cc} G'' &
    d(G'')\\ d(G'') & G_{33}\end{array}\right] \Big\rangle \\ & =
\Big\langle\Big( E + \mathcal{L}_3 (\left[\begin{array}{cc} 0 &
    d(G'') - G_{13, 23}\\ d(G'') - G_{13, 23}& 0\end{array}\right])\Big) : \left[\begin{array}{cc} G'' &
    d(G'')\\ d(G'') & G_{33}\end{array}\right] \Big\rangle \\ &
= \langle E'': G''\rangle + E_{33}G_{33} \\
&\qquad\qquad\quad + \Big\langle\mathcal{L}_3
(\left[\begin{array}{cc} G'' &
    d(G'')\\ d(G'') & G_{33}\end{array}\right]) : \left[\begin{array}{cc} 0 &
    d(G'') - G_{13, 23}\\ d(G'') - G_{13, 23}& 0\end{array}\right]
\Big\rangle\\ & = \langle E'': G''\rangle + E_{33}G_{33} = \langle
E: G\rangle = \mathcal{Q}_3(G),
\end{split}
\end{equation*}
where we repeatedly used the assumptions on $G$ and $E$, and
(\ref{dmin}). Consequently, by uniqueness of the minimizer $d$, it
follows that:
\begin{equation}\label{Ggood}
d(G'') = G_{13, 23}.
\end{equation}
Take any $F''\in \mathbb{R}^{2\times 2}$. By (\ref{dmin}) and
(\ref{lind}), we see that:
$$ \mathcal{Q}_2(G''+F'') = \mathcal{Q}_3(\left[\begin{array}{cc} G''+
    F'' & d(G'') + d(F'')\\ d(G'') + d(F'') & G_{33} - \mbox{Tr } F''\end{array}\right]). $$
Expanding the above and removing $\mathcal{Q}_2(G'')$ and
$\mathcal{Q}_2(F'')$ from both sides, we obtain:
\begin{equation*}
\begin{split}
\langle\mathcal{L}_2(G'') : F''\rangle & = \Big\langle \mathcal{L}_3
(\left[\begin{array}{cc} G'' &
    d(G'')\\ d(G'') & -\mbox{Tr } G''\end{array}\right]) : \left[\begin{array}{cc} F'' &
    d(F'')\\ d(F'') & -\mbox{Tr } F''\end{array}\right] \Big\rangle \\
& = \Big\langle \mathcal{L}_3 (\left[\begin{array}{cc} G'' &
    d(G'')\\ d(G'') & -\mbox{Tr } G''\end{array}\right]) : \left[\begin{array}{cc} F'' &
    0\\ 0 & -\mbox{Tr } F''\end{array}\right] \Big\rangle \\ &
= \Big\langle \mathcal{L}_3 (G) : \left[\begin{array}{cc} F'' &
    d(F'')\\ d(F'') & -\mbox{Tr } F''\end{array}\right] \Big\rangle =
\Big\langle E : \left[\begin{array}{cc} F'' &
    d(F'')\\ d(F'') & -\mbox{Tr } F''\end{array}\right] \Big\rangle \\
& = \langle E'' - E_{33}\mbox{Id}_2 : F''\rangle,
\end{split}
\end{equation*}
by (\ref{Ggood}) and assumptions on $E$ and $G$. The expression
(\ref{L2form}) follows now directly.
\end{proof}

In section \ref{secGE} below we shall prove that for almost every
$x\in\Omega^1$ there holds:
\begin{equation}\label{property}
\mbox{Tr } G(x) = 0 \quad \mbox{and} \quad E_{13}(x) = E_{23}(x)=0.
\end{equation}
Therefore, recalling Lemma \ref{lemMS} (iii), we observe that the
limiting stress and strain satisfy the assumptions of Lemma
\ref{lemma4.1} pointwise almost everywhere. We now record the
following simple conclusion which will be used in deriving the
Euler-Lagrange equations (\ref{EL1}), (\ref{EL2}).

\begin{lemma}
Let $E, G\in L^2(\Omega^1,\mathbb{R}^{3\times 3})$ be the limiting
strain and stress as in Lemma \ref{lemMS}, which are related to
$(w,u)$ by (\ref{Lim_Gh}). Then, for almost every $x'\in\Omega$,
there holds:
\begin{equation}\label{aiuto}
\begin{split}
& \int_{-1/2}^{1/2} (E'' -E_{33}\mathrm{Id}_2) ~\mathrm{d}x_3 =
\mathcal{L}_2^{in}\left(\mathrm{sym} \nabla w + \frac{1}{2}\nabla
  v\otimes \nabla v\right),\\
& \int_{-1/2}^{1/2} x_3(E'' -E_{33}\mathrm{Id}_2)~\mathrm{d}x_3 = -
\frac{1}{12}\mathcal{L}_2^{in}\left(\nabla^2v\right).
\end{split}
\end{equation}
\end{lemma}
\begin{proof}
By Lemma \ref{TrG=0}, Lemma \ref{E12=0}, Lemma \ref{lemma4.1} and
(\ref{Lim_Gh})  we see that:
\begin{equation*}
\begin{split}
& \int_{-1/2}^{1/2} (E'' -E_{33}\mathrm{Id}_2) ~\mathrm{d}x_3 =
\int_{-1/2}^{1/2} \mathcal{L}_2^{in}(G'')~\mathrm{d}x_3  \\
& \qquad\qquad  = \mathcal{L}_2^{in}\left( \int_{-1/2}^{1/2} G''(x',
x_3)~\mathrm{d}x_3
\right) = \mathcal{L}_2^{in} (G_0(x')) = \mathcal{L}_2^{in} (\mbox{sym }G_0(x'))\\
& \int_{-1/2}^{1/2} x_3(E'' -E_{33}\mathrm{Id}_2)~\mathrm{d}x_3 =
\int_{-1/2}^{1/2} x_3\mathcal{L}_2^{in}(G'')~\mathrm{d}x_3 \\
& \qquad\qquad = \mathcal{L}_2^{in}\left( \int_{-1/2}^{1/2}
x_3G''(x',
  x_3)~\mathrm{d}x_3\right)
= - \mathcal{L}_2^{in}\left( \int_{-1/2}^{1/2}
  x_3^2G_1(x')~\mathrm{d}x_3\right)
= - \frac{1}{12}\mathcal{L}_2^{in}(G_1(x')).
\end{split}
\end{equation*}
This concludes the proof, in view of (\ref{Lim_Gh}).
\end{proof}

\section{Two further properties of $G$ and $E$}\label{secGE}

In this section we derive the two fundamental properties of the
incompressible stress and strain, allowing for pointwise application
of Lemma \ref{lemma4.1}, and ultimately leading to formulas in
(\ref{aiuto}).

\begin{lemma}\label{TrG=0}
The limiting strain $G(x)$ is traceless, for almost every
$x\in\Omega^1$.
\end{lemma}
\begin{proof}
Recall that $\nabla u^h(x', hx_3) = R^h(x') \big(\mbox{Id} + h^2
G^h(x', x_3)\big)$. Therefore:
$$1 = \det\nabla u^h = \det(\mbox{Id} + h^2G^h) = 1 + h^2\mbox{Tr } G^h
+ h^4 \mbox{Tr } \mbox{cof } G^h + h^6 \det G^h,$$ and consequently:
\begin{equation}\label{6}
\mbox{Tr } G^h + h^2\mbox{Tr } \mbox{cof } G^h + h^4 \det G^h = 0.
\end{equation}
Fix an exponent $\gamma\in (\frac{2}{3}, 2)$ and define $B_h =
\{x\in\Omega^1;~ h^{2-\gamma} |G^h(x) \leq 1\}$ as in Lemma
\ref{lemMS} (iv). Then:
\begin{equation*}
\begin{split}
\int_{\Omega^1\setminus B_h} |h^4\det G^h| & =
\int_{\Omega^1\setminus B_h} |\mbox{Tr } G^h + h^2\mbox{Tr }
\mbox{cof } G^h | \\ & \leq |\Omega^1\setminus B_h|^{1/2} \left(
\int_{\Omega^1\setminus B_h} |\mbox{Tr } G^h|^2\right)^{1/2} + h^2
\int_{\Omega^1} |\mbox{Tr } \mbox{cof } G^h| \leq C (h^{2-\gamma} +
h^2),
\end{split}
\end{equation*}
where we used (\ref{quattro}) and the boundedness of $G^h$ in
$L^2(\Omega^1)$. On the other hand, we have:
$$\int_{B_h} |h^4\det G^h| = \frac{h^4}{h^{6-3\gamma}} \int_{B_h}
|\det (h^{2-\gamma} G^h)| \leq C h^{3\gamma-2}.$$ Hence, by
(\ref{6}) and, again the boundedness of $\mbox{Tr } \mbox{cof } G^h$
in $L^1(\Omega^1)$, it follows that:
$$\int_{\Omega^1} |\mbox{Tr } G^h| \leq \int_{\Omega^1} |h^2\mbox{Tr }
\mbox{cof } G^h| + \int_{\Omega^1} |h^4 \det G^h| \to 0, \quad
\mbox{as } h\to 0.$$ Observing that $\mbox{Tr } G^h\rightharpoonup
\mbox{Tr } G$ weakly in $L^2(\Omega^1)$, we conclude that $\mbox{Tr
} G = 0$.
\end{proof}

\medskip

We now prove the remaining property of the strain $E$ in
(\ref{property}). The strategy of proof is the same as in the later
proofs of the Euler-Lagrange equations; we will apply the
equilibrium equation (\ref{Eq2}) to appropriate test functions
$\phi^h$, such that after passing to the limit with $h\to 0$ only
some chosen terms will survive, yielding the week formulation of
(\ref{property}). One difficulty with (\ref{Eq2}) is that it only
allows for globaly bounded $\phi^h$. For this reason, following
\cite{MS}, we introduce a family of truncation functions $\theta^h$
which coincide with the identity on intervals $(-\omega_h,
\omega_h)$ with a suitable rate of convergence of
$\omega_h\to\infty$.

\begin{lemma}\label{thetah}
Let $\{\omega_h\}$ be a sequence of positive numbers, increasing to
$+\infty$ as $h\to 0$. There exists a sequence of nondecreasing
functions $\theta^h\in\mathcal{C}_b^2(\mathbb{R},\mathbb{R})$ with
the following properties:
\begin{equation}\label{theta_prop}
\begin{split}
&\theta^h(t) = t \quad \forall |t|\leq \omega_h \qquad \mbox{and}
\qquad \theta^h(t) = (\mathrm{sgn}~t) \frac{3}{2}\omega_h\quad
\forall |t|\geq
2 \omega_h \\
& |\theta^h(t)|\leq t \quad \forall t \qquad \mbox{and}
\qquad \|\theta^h\|_{L^\infty} \leq \frac{3}{2}\omega_h\\
&  \|\frac{\mathrm{d}}{\mathrm{d}t}{\theta^h}\|_{L^\infty} \leq 1
\qquad \mbox{and} \qquad
\|\frac{\mathrm{d^2}}{\mathrm{d}t^2}{\theta^h} \|_{L^\infty} \leq
\frac{C}{\omega_h}.
\end{split}
\end{equation}
\end{lemma}
\begin{proof}
One may take:
\begin{equation*}
\theta^h(t) = \left\{\begin{array}{ll} t & \mbox{for } |t|\leq
    \omega_h\\ \displaystyle{(\mbox{sgn } t)
\frac{1}{2}\left(|t|+ \omega_h
+\frac{\omega_h}{\pi}\sin\left(\frac{\pi
      |t|-\omega_h}{\omega_h}\right)\right)} & \mbox{for } |t|\in[\omega_h,
2\omega_h]\\
\displaystyle{(\mbox{sgn } t) \frac{3}{2}\omega_h} & \mbox{for }
|t|\geq  \omega_h
\end{array}\right.
\end{equation*}
\end{proof}

\begin{lemma}\label{E12=0}
The limiting stress $E(x)$ satisfies: $E_{13}(x)=E_{23}(x)= 0$ for
almost every $x\in \Omega^1$.
\end{lemma}
\begin{proof}
{\bf 1.} Let $\eta=(\eta_1,
\eta_2)\in\mathcal{C}_b^2(\mathbb{R}^3,\mathbb{R}^2)$ be a given
test function, and define:
\begin{equation}\label{fuer}
\eta_3(x', x_3) = -\int_0^{x_3} \mbox{div }\eta(x', s)~\mbox{d}s.
\end{equation}
Since $\partial_3\eta_3 = -\mbox{div }\eta$, the following test
functions $\phi^h\in\mathcal{C}^1_{b}(\mathbb{R}^3,\mathbb{R}^3)$
are  divergence-free:
\begin{equation*}
\phi^h(x', x_3) = \left[\begin{array}{c}
\displaystyle{h{\theta^h}'\left(\frac{x_3}{h}\right)
    \eta \left(x', \theta^h\left(\frac{x_3}{h}\right)\right)}\vspace{2mm}\\
\displaystyle{h^2 \eta_3 \left(x',
\theta^h\left(\frac{x_3}{h}\right)\right)}\end{array}\right],
\end{equation*}
and denoting $\nabla_{tan}$ the gradient in the tangential
directions $e_1, e_2$, we have:
\begin{equation*}
\nabla\phi^h(x', x_3) = \left[\begin{array}{c|c}
\displaystyle{h{\theta^h}'\left(\frac{x_3}{h}\right)
    \nabla_{tan}\eta \left(x', \theta^h\left(\frac{x_3}{h}\right)\right)} &
\begin{minipage}{6cm}\[\displaystyle{ \left({\theta^h}'\left(\frac{x_3}{h}\right)\right)^2
\partial_3\eta \left(x', \theta^h\left(\frac{x_3}{h}\right)\right)} \] \[ \qquad \qquad + ~
\displaystyle{ {\theta^h}''\left(\frac{x_3}{h}\right) \eta \left(x',
\theta^h\left(\frac{x_3}{h}\right)\right)}\]\vspace{2mm}\end{minipage}\\\hline
\displaystyle{h^2 \nabla_{tan} \eta_3 \left(x',
\theta^h\left(\frac{x_3}{h}\right)\right)} &
\begin{minipage}{6cm}\[\vspace{2mm}\displaystyle{h {\theta^h}'\left(\frac{x_3}{h}\right)
\partial_3\eta_3 \left(x', \theta^h\left(\frac{x_3}{h}\right)\right)}\]\end{minipage}
\end{array}\right].
\end{equation*}
The truncations $\theta^h$ are chosen as in Lemma \ref{thetah} and
such that:
\begin{equation}\label{omegah}
\lim_{h\to 0}\omega_h = +\infty \quad \mbox{ and } \quad
h^2\omega_h\leq C.
\end{equation}

\medskip

{\bf 2.} Applying the equilibrium equation (\ref{Eq2}) with $\phi =
\phi^h$, we obtain:
\begin{equation}\label{bigeq1}
\begin{split}
 h\int_{\Omega^1}\Big\langle & \Big((\bar R^h)^T R^h E^h  (R^h)^T \bar
R^h\Big)'' - \Big((\bar R^h)^T R^h E^h (R^h)^T \bar
R^h\Big)_{33}\mbox{Id}_2 : {\theta^h}'\left(\frac{y^h_3}{h}\right)
    \nabla_{tan}\eta ({y^h}', \theta^h\left(\frac{y_3^h}{h}\right))\Big\rangle\\
& \qquad\qquad + \int_{\Omega^1}  \Big\langle \Big((\bar R^h)^T R^h
E^h (R^h)^T \bar R^h\Big)_{13,23} ,
({\theta^h}'\left(\frac{y^h_3}{h}\right))^2
    \partial_3\eta ({y^h}', \theta^h\left(\frac{y_3^h}{h}\right))\Big\rangle\\
& \qquad\qquad + \int_{\Omega^1}  \Big\langle \Big((\bar R^h)^T R^h
E^h (R^h)^T \bar R^h\Big)_{13,23} ,
{\theta^h}''\left(\frac{y^h_3}{h}\right)
    \eta ({y^h}', \theta^h\left(\frac{y_3^h}{h}\right))\Big\rangle\\
& \qquad\qquad + h^2 \int_{\Omega^1}  \Big\langle \Big((\bar R^h)^T
R^h E^h (R^h)^T \bar
R^h\Big)_{31,32} , \nabla_{tan}\eta_3 ({y^h}', \theta^h\left(\frac{y_3^h}{h}\right))\Big\rangle\\
& = h^2\int_{\Omega^1} \Big\langle f(x') (\bar R^h)_{31,32} , \eta
({y^h}', \theta^h\left(\frac{y_3^h}{h}\right))\Big\rangle +
h^3\int_{\Omega^1} f(x')  (\bar R^h)_{33} \eta_3 ({y^h}',
\theta^h\left(\frac{y_3^h}{h}\right)).
\end{split}
\end{equation}

Now, we will discuss the convergence as $h\to 0$ of each term in
(\ref{bigeq1}). The first term converges to $0$, because $\Big((\bar
R^h)^T R^h E^h  (R^h)^T \bar R^h\Big)'' - \Big((\bar R^h)^T R^h E^h
(R^h)^T \bar R^h\Big)_{33}\mbox{Id}_2$ is bounded in $L^1(\Omega^1)$
in view of Lemma \ref{lemMS} (iii), while $
{\theta^h}'\left(\frac{y^h_3}{h}\right)
    \nabla_{tan}\eta ({y^h}', \theta^h\left(\frac{y_3^h}{h}\right))$
    is pointwise bounded by (\ref{theta_prop}).

\medskip

{\bf 3.}  The second term in (\ref{bigeq1}) when integrated over
$\Omega^1\setminus B_h$, goes to $0$ in view of (\ref{quattro}) and
of the pointwise boundedness of
$({\theta^h}'\left(\frac{y^h_3}{h}\right))^2
    \partial_3\eta ({y^h}', \theta^h\left(\frac{y_3^h}{h}\right))$ by
(\ref{theta_prop}). On the other hand, the limit of this integral
over $B_h$ is the same as the limit of:
\begin{equation}\label{help}
\int_{\Omega^1}  \Big\langle \chi_h E^h_{13, 23} ,
({\theta^h}'\left(\frac{y^h_3}{h}\right))^2
    \partial_3\eta ({y^h}',
    \theta^h\left(\frac{y_3^h}{h}\right))\Big\rangle~\mbox{d}x
\end{equation}
because of (\ref{due}). We now conclude that the integrals in
(\ref{help}) converge to:
$$\int_{\Omega^1}  \Big\langle E_{13, 23} , \partial_3\eta (x',
   x_3+ v(x'))\Big\rangle~\mbox{d}x.$$
This follows by recalling (\ref{cinque}) and observing that:
\begin{equation}\label{zwei}
({\theta^h}'\left(\frac{y^h_3}{h}\right))^2
    \partial_3\eta ({y^h}', \theta^h\left(\frac{y_3^h}{h}\right))
    \rightarrow
\partial_3\eta(x', x_3+ v(x')) \quad \mbox{ in } L^2(\Omega^1)
\end{equation}
Indeed:
\begin{equation*}
\begin{split}
\int_{\Omega^1}&\left|({\theta^h}'\left(\frac{y^h_3}{h}\right))^2
    \partial_3\eta ({y^h}', \theta^h\left(\frac{y_3^h}{h}\right))
    - \partial_3\eta(x', x_3+ v(x')) \right|^2 ~\mbox{d}x \\ & \leq
C\int_{\Omega^1}
\left|{\theta^h}'\left(\frac{y^h_3}{h}\right)\right|^4 \left(|{y^h}'
- x'|^2 + \left|\theta^h\left(\frac{y_3^h}{h}\right) -
    (x_3 + v(x'))\right|^2\right) ~\mbox{d}x  \\ &
\qquad + C\int_{\Omega^1}
\left|{\theta^h}'\left(\frac{y^h_3}{h}\right) - 1\right|^2
~\mbox{d}x \\
& \leq C \int_{\Omega^1} |{y^h}' - x'|^2 + \left|\frac{y_3^h}{h} -
    (x_3 + v(x'))\right|^2 ~\mbox{d}x
+ C \int_{\left\{x\in\Omega^1;~ \frac{|y^h_3|}{h}\geq
\omega_h\right\}} 1+ \left|\frac{y_3^h}{h} \right|^2 ~\mbox{d}x
\end{split}
\end{equation*}
converges to $0$ as $h\to 0$, by (\ref{convi}), (\ref{convi2}) and
(\ref{omegah}), proving hence (\ref{zwei}).

\medskip

{\bf 4.} The third term in (\ref{bigeq1}) is bounded by:
$\frac{C}{\omega_h} \int_{\Omega^1} |E^h|$ by (\ref{theta_prop}). It
therefore converges to $0$ in view of the boundedness of $E^h$ in
$L^1(\Omega^1)$ and (\ref{omegah}).

The fourth term in (\ref{bigeq1}) is bounded by:
\begin{equation*}
\begin{split}
Ch^2\int_{\Omega^1} |E^h|
\left|\theta^h\left(\frac{y^h_3}{h}\right)\right|~\mbox{d}x & \leq
Ch^2\omega_h\int_{\Omega^1\setminus B_h} |E^h| + Ch^2\int_{\Omega^1}
\chi_h|E^h| \frac{|y^h_3|}{h} \\ & \leq Ch^2\omega_h~ {o}(1) + Ch^2
\|\chi_h E^h\|_{L^2(\Omega^1)} \left\| \frac{y^h_3}{h}  \right
\|_{L^2(\Omega^1)},
\end{split}
\end{equation*}
and it converges to $0$ by (\ref{quattro}), (\ref{cinque}),
(\ref{omegah}) and the boundedness of $\frac{y^h_3}{h}$ in
$L^2(\Omega^1)$.

Finally, both terms in the right hand side of (\ref{bigeq1}) are
bounded by:
\begin{equation*}
\begin{split}
Ch^2\int_{\Omega^1} |f(x')|
\left(\left|{\theta^h}'\left(\frac{y^h_3}{h}\right)\right| + h
\left|\theta^h\left(\frac{y^h_3}{h}\right)\right| \right) ~\mbox{d}x
\leq Ch^2\int_{\Omega^1} |f(x')| (1 + h\omega_h) ~\mbox{d}x\leq Ch
\|f\|_{L^2(\Omega)},
\end{split}
\end{equation*}
which clearly converges to $0$. Above, we used (\ref{theta_prop})
and (\ref{omegah}).

\medskip

{\bf 5.} In conclusion, passing to the limit with $h\to 0$ in
(\ref{bigeq1}), results in:
\begin{equation}\label{drei}
\int_{\Omega^1}  \Big\langle E_{13, 23} , \partial_3\eta (x',
   x_3+ v(x'))\Big\rangle~\mbox{d}x = 0 \qquad\forall
   \eta\in\mathcal{C}^2_b(\mathbb{R}^3, \mathbb{R}^2).
\end{equation}
We now reproduce an argument from \cite{MS}, in order  to deduce
that $E_{13, 23}= 0.$ Take an arbitrary
$\phi\in\mathcal{C}^2_c(\Omega, \mathbb{R}^2).$ Let
$\mathcal{C}_c^\infty(\Omega, \mathbb{R})\ni  {v}_k\rightarrow v$ in
$L^2(\Omega)$, and define:
$$ \phi_k(x',x_3) = \phi(x', x_3 - v_k(x')), \qquad \eta(x', x_3) =
\int_0^{x_3} \phi_k(x', s) ~\mbox{d}s$$ Clearly $\phi_k\in
\mathcal{C}^2_c(\mathbb{R}^3, \mathbb{R}^2)$, $ \eta\in
\mathcal{C}^2_b(\mathbb{R}^3, \mathbb{R}^2)$, and thus by
(\ref{drei}) we obtain:
$$ 0 = \int_{\Omega^1}  \Big\langle E_{13, 23} , \phi_k (x',
   x_3+ v(x'))\Big\rangle~\mbox{d}x =  \int_{\Omega^1}  \Big\langle E_{13, 23} , \phi (x',
   x_3+ v(x')-v_k(x'))\Big\rangle~\mbox{d}x $$
Passing to the limit with $k\to\infty$, it follows that:
\begin{equation*}
\int_{\Omega^1}  E_{13, 23}  \phi (x', x_3)~\mbox{d}x = 0
\qquad\forall
   \phi\in\mathcal{C}^2_c(\Omega, \mathbb{R}^2)
\end{equation*}
which concludes the proof.
\end{proof}

\section{Derivation of the first Euler-Lagrange equation (\ref{EL1})}

{\bf 1.} Let $\eta=(\eta_1,
\eta_2)\in\mathcal{C}_b^2(\mathbb{R}^2,\mathbb{R}^2)$ be a given
test function, and let $\eta_3(x')=-\mbox{div }\eta(x')$. Given
$\theta^h$ as in Lemma \ref{thetah}, with:
\begin{equation}\label{omegah2}
\lim_{h\to 0}\omega_h = \lim_{h\to 0} h\omega_h^2= +\infty \quad
\mbox{ and } \quad h\omega_h \leq C,
\end{equation}
consider the following divergence-free test functions
$\phi^h\in\mathcal{C}^1_{b}(\mathbb{R}^3,\mathbb{R}^3)$:
\begin{equation*}
\phi^h(x', x_3) = \left[\begin{array}{c}
\displaystyle{{\theta^h}'\left(\frac{x_3}{h}\right)
    \eta (x')}\vspace{2mm}\\
\displaystyle{h {\theta^h}\left(\frac{x_3}{h}\right)\eta_3
(x')}\end{array}\right],
\end{equation*}
Denoting $\nabla_{tan}$ the gradient in the tangential directions
$e_1, e_2$, we have:
\begin{equation*}
\nabla\phi^h(x', x_3) = \left[\begin{array}{c|c}
\displaystyle{{\theta^h}'\left(\frac{x_3}{h}\right)
    \nabla_{tan}\eta (x')} & \begin{minipage}{3cm}\[\vspace{3mm}
\displaystyle{ \frac{1}{h}{\theta^h}''\left(\frac{x_3}{h}\right)
\eta (x')} \]\end{minipage} \\ \hline \displaystyle{h
{\theta^h}\left(\frac{x_3}{h}\right) \nabla_{tan} \eta_3 (x') }&
\begin{minipage}{3cm}\[\displaystyle{ {\theta^h}'\left(\frac{x_3}{h}\right)
\eta_3 (x')}\]\end{minipage}
\end{array}\right].
\end{equation*}

\medskip

{\bf 2.} Applying the equilibrium equation (\ref{Eq2}) with $\phi =
\phi^h$, we obtain:
\begin{equation}\label{bigeq2}
\begin{split}
 \int_{\Omega^1}\Big\langle & \Big((\bar R^h)^T R^h E^h  (R^h)^T \bar
R^h\Big)'' - \Big((\bar R^h)^T R^h E^h (R^h)^T \bar
R^h\Big)_{33}\mbox{Id}_2 : {\theta^h}'\left(\frac{y^h_3}{h}\right)
    \nabla_{tan}\eta ({y^h}')\Big\rangle\\
& \qquad\qquad + h \int_{\Omega^1}  \Big\langle \Big((\bar R^h)^T
R^h E^h (R^h)^T \bar R^h\Big)_{31,32} ,
{\theta^h}\left(\frac{y^h_3}{h}\right)
    \nabla_{tan}\eta_3 ({y^h}')\Big\rangle\\
& \qquad\qquad + \frac{1}{h} \int_{\Omega^1}  \Big\langle \Big((\bar
R^h)^T R^h E^h (R^h)^T \bar R^h\Big)_{13,23} ,
{\theta^h}''\left(\frac{y^h_3}{h}\right)
    \eta ({y^h}')\Big\rangle\\
& = h\int_{\Omega^1} \left\langle f(x') (\bar R^h)_{31,32} ,
{\theta^h}'\left(\frac{y^h_3}{h}\right) \eta
({y^h}')\right\rangle~\mbox{d}x +
 h^2\int_{\Omega^1} f(x') (\bar R^h)_{33} {\theta^h}\left(\frac{y^h_3}{h}\right)
\eta_3 ({y^h}')~\mbox{d}x.
\end{split}
\end{equation}

Now, we will check convergence as $h\to 0$ of each of the four terms
in the identity (\ref{bigeq2}). Regarding the first term, it
converges to $0$ when integrated over $\Omega^1\setminus B_h$, by
(\ref{quattro}) and by the pointwise boundedness of
${\theta^h}'\left(\frac{y^h_3}{h}\right)
    \nabla_{tan}\eta ({y^h}')$ in view of (\ref{theta_prop}). On the other
hand, the limit of this integral over $B_h$ is the same as the limit
of:
\begin{equation}\label{help2}
\int_{\Omega^1} \Big\langle \chi_h \big({E^h}'' -
E^h_{33}\mbox{Id}_2\big) : {\theta^h}'\left(\frac{y^h_3}{h}\right)
    \nabla_{tan}\eta ({y^h}')\Big\rangle~\mbox{d}x,
\end{equation}
because of the convergence in (\ref{due}). Now, the limit of
integrals in (\ref{help2}) equals:
$$\int_{\Omega^1} \Big\langle {E}'' -
E_{33}\mbox{Id}_2 : \nabla\eta (x')\Big\rangle~\mbox{d}x,$$ in view
of (\ref{cinque}) and:
\begin{equation*}
\begin{split}
\int_{\Omega^1}  &\left|{\theta^h}'\left(\frac{y^h_3}{h}\right)
    \nabla_{tan}\eta ({y^h}') - \nabla\eta(x')\right|^2~\mbox{d}x
\\ & \qquad\qquad \leq
C\int_{\Omega^1}  \left|\nabla_{tan}\eta ({y^h}') -
  \nabla\eta(x')\right|^2 + C\int_{\Omega^1}
\left|{\theta^h}'\left(\frac{y^h_3}{h}\right) -1\right|^2
\\ & \qquad \qquad \leq
C\int_{\Omega^1}  \left| {y^h}' - x'\right|^2~\mbox{d}x +
C\left|\left\{
  x\in\Omega^1;~ \frac{|y^h_3(x)|}{h}\geq \omega_h\right\}\right|
\\ & \qquad \qquad \leq
C\int_{\Omega^1}  \left| {y^h}' - x'\right|^2~\mbox{d}x +
\frac{C}{\omega_h^2},
\end{split}
\end{equation*}
where we apply (\ref{convi2}), and then (\ref{convi2}) and
(\ref{omegah2}) to conclude the convergence of both terms in the
right hand side of the above displayed expression to $0$.

\medskip

{\bf 3.} The second term in (\ref{bigeq2}) is bounded by:
\begin{equation*}
\begin{split}
Ch \int_{\Omega^1\setminus B_h}
{\theta^h}&\left(\frac{|y^h_3|}{h}\right) |E^h|~\mbox{d}x  + Ch
\int_{\Omega^1} |\chi_hE^h| \frac{|y^h_3|}{h} ~\mbox{d}x \\ & \leq
Ch\omega_h \int_{\Omega^1\setminus B_h} |E^h|~\mbox{d}x + C
\|y^h_3\|_{L^2(\Omega^1)} \|\chi_h E^h\|_{L^2(\Omega^1)}
\end{split}
\end{equation*}
and it clearly converges to $0$ by (\ref{quattro}), (\ref{cinque}),
(\ref{convi}) and (\ref{omegah2}).

The third term in (\ref{bigeq2}) is bounded by:
\begin{equation*}
\begin{split}
& \frac{C}{h\omega_h}\int_{\left\{
  x\in\Omega^1;~ \frac{|y^h_3(x)|}{h}\geq \omega_h\right\}} |E^h|
~\mbox{d}x \leq \frac{C}{h\omega_h}\int_{\left\{
  x\in\Omega^1;~ \frac{|y^h_3(x)|}{h}\geq \omega_h\right\}}
\frac{1}{h^2} W(\mbox{Id} + h^2G^h) + |G^h| ~\mbox{d}x \\ & \qquad
\leq \frac{C}{h^3\omega_h} \int_{\Omega^1} W\big(\nabla u^h(x',
hx_3)\big) ~\mbox{d}x + \frac{C}{h\omega_h} \|G^h\|_{L^2(\Omega^1)}
\left|\left\{
  x\in\Omega^1;~ \frac{|y^h_3(x)|}{h}\geq \omega_h\right\}\right|^{1/2}
\\ & \qquad \leq C\left(\frac{h}{\omega_h} + \frac{1}{h\omega_h^2}\right),
\end{split}
\end{equation*}
by (\ref{tre}), (\ref{convi2}), the boundedness of $G^h$ in
$L^2(\Omega^1)$ and (\ref{en_as}). Then, the right hand side above
converges to $0$ by (\ref{omegah2}).

Finally, the right hand side of (\ref{bigeq2}) converges to $0$ as
well, as it is bounded by:
$$Ch\int_{\Omega^1} |f(x')| (1+h\omega_h) ~\mbox{d}x
  \leq  Ch \|f\|_{L^2(\Omega)}.$$
In conclusion, passing to the limit with $h\to 0$ in (\ref{bigeq2})
we obtain:
\begin{equation}\label{krak0}
\int_{\Omega^1} \Big\langle {E}'' - E_{33}\mbox{Id}_2 : \nabla\eta
(x')\Big\rangle~\mbox{d}x=0 \quad \forall
\eta\in\mathcal{C}^2_b(\mathbb{R}^2, \mathbb{R}^2).
\end{equation}
and thus the Euler-Lagrange equation (\ref{EL1}) follows directly,
in view of (\ref{aiuto}) and the density of test functions $\eta$ as
above in $W^{1,2}(\Omega,\mathbb{R}^2)$.
\endproof

\section{Derivation of the second Euler-Lagrange equation (\ref{EL2})}

\begin{lemma}\label{lem7.1}
For every $\eta_3\in\mathcal{C}_b^3(\mathbb{R}^2,\mathbb{R})$, it
follows that:
\begin{equation}\label{Lbigeq3}
\begin{split}
\int_{\Omega^1}\Big\langle (E'' - E_{33}\mathrm{Id}_{2}) :
  \nabla v \otimes \nabla \eta_3\Big\rangle \mathrm{d}x
~+ ~& \lim_{h \to 0}\frac{1}{h}\int_{\Omega^1}\left\langle E^h_{31,
32},
  \nabla \eta_3 ({y^h}')\right \rangle \mathrm{d}x \\
& \qquad = \bar R_{33}\int_{\Omega}f(x')\eta_3(x')~\mathrm{d}x'.
\end{split}
\end{equation}
\end{lemma}
\begin{proof} {\bf 1.} Given $\eta_3\in\mathcal{C}_b^3(\mathbb{R}^2,\mathbb{R})$
consider the divergence-free test functions
$\phi^h\in\mathcal{C}_b^!(\mathbb{R}^3,\mathbb{R}^3)$:
\begin{equation*}
\phi^h(x', x_3) = \left[\begin{array}{c} 0 \vspace{2mm}\\
\displaystyle{\frac{1}{h} \eta_3 (x')}\end{array}\right], \quad
\mbox{so that} \quad \nabla\phi^h(x', x_3) =
\left[\begin{array}{c|c} 0 & \begin{minipage}{1cm}\[\vspace{3mm} 0
\]\end{minipage} \\ \hline \displaystyle{\frac{1}{h} \nabla_{tan}
\eta_3 (x') }&
\begin{minipage}{1cm}\[ \vspace{3mm}  0\]\end{minipage}
\end{array}\right].
\end{equation*}
Applying the equilibrium equation (\ref{Eq2}) with $\phi = \phi^h$,
we obtain:
\begin{equation}\label{bigeq3}
 \frac{1}{h}\int_{\Omega^1}\Big\langle \Big((\bar R^h)^T R^h E^h  (R^h)^T \bar
R^h\Big)_{31, 32} , \nabla_{tan}\eta_3 ({y^h}')\Big\rangle
~\mbox{d}x = \bar R_{33}^h\int_{\Omega^1} f(x') \eta_3
({y^h}')~\mbox{d}x.
\end{equation}
Recall that the tensor field $A^h$ in (\ref{Ah}) is defined as: $A^h
(x') = \frac{1}{h}\left((\bar R^h)^T R^h(x') - \mbox{Id}\right)$.
Hence:
\begin{equation}\label{RERA}
\begin{split}
\frac{1}{h}(\bar R^h)^T R^hE^h(R^h)^T \bar R^h = A^hE^h(R^h)^T \bar
R^h + E^h(A^h)^T + \frac{1}{h}E^h,
\end{split}
\end{equation}
and therefore the left hand side of (\ref{bigeq3}) can be written
as:
\begin{equation}\label{momo}
\begin{split}
\int_{\Omega^1}&\left\langle (A^hE^h(R^h)^T \bar R^h)_{31, 32}, \nabla \eta_3({y^h}')\right\rangle ~\mbox{d}x\\
&+ \int_{\Omega^1}\left\langle (E^h(A^h)^T)_{31, 32}, \nabla \eta_3
  ({y^h}')\right\rangle  ~\mbox{d}x
+ \frac{1}{h} \int_{\Omega^1} \left\langle E_{31, 32}^h, \nabla
\eta_3 ({y^h}')\right\rangle ~\mbox{d}x.
\end{split}
\end{equation}

\medskip

{\bf 2.} Let the sets $B_h$ be defined as in Lemma \ref{lemMS} (iv),
for some exponent $\gamma\in (0,1)$. The first two terms in
(\ref{momo}), when considered on $\Omega^1\setminus B_h$, converge
to $0$ because they are bounded by:
$$C\int_{\Omega^1\setminus B_h} |A^h| |E^h| ~\mbox{d}x
\leq \frac{C}{h} \int_{\Omega^1\setminus B_h} |E^h| ~\mbox{d}x \leq
\frac{C}{h} h^{2-\gamma},$$ in view of (\ref{quattro}) and
$|A^h|\leq\frac{C}{h}$. On the other hand, the same two terms while
on $B_h$, converge to:
$$\int_{\Omega^1} \left\langle (AE)_{31, 32} ,
  \nabla\eta_3(x')\right\rangle + \left\langle (EA^T)_{31, 32} ,
  \nabla\eta_3(x')\right\rangle ~\mbox{d}x,$$
where we used the convergence (\ref{cinque}) and the following
strong convergences in $L^{3}(\Omega^1)$: of $A^h$ to $A$ by
(\ref{CAh}), of $(R^h)^T\bar R^h$ to $\mbox{Id}$ by (\ref{due}), and
of $\nabla\eta_3({y^h}')$ to $\nabla \eta_3(x')$ in view of the
Sobolev embedding and the strong convergence in $W^{1,2}(\Omega^1,
\mathbb{R}^2)$ in (\ref{convi}).

Concluding, the first two terms in (\ref{momo}) converge to:
$$\int_{\Omega^1} \Big\langle E'' \nabla v , \nabla\eta_3(x')\Big\rangle
- \Big\langle E_{33} \nabla v , \nabla\eta_3(x')\Big\rangle
~\mbox{d}x$$ in view of the structure of the limiting tensor $A$ in
(\ref{CAh}). Since the right hand side of (\ref{bigeq3}) converges
to $\bar R_{33} \int_\Omega f(x') \eta_3(x')$ by (\ref{convi}),
passing to the limit in all terms of  (\ref{bigeq3}) yields the
desired equality (\ref{Lbigeq3}) and thus proves the lemma.
\end{proof}

\begin{lemma}\label{lem7.2}
For every $\eta\in\mathcal{C}_b^2(\mathbb{R}^2,\mathbb{R}^2)$, it
follows that:
\begin{equation}\label{Lbigeq4}
\begin{split}
\int_{\Omega^1}&\Big\langle (E'' - E_{33}\mathrm{Id}_{2})
  : (x_3 + v(x'))\nabla_{tan}\eta(x')\Big\rangle ~\mathrm{d}x \\ &
+ \int_{\Omega^1}\Big\langle (E'' - E_{33}\mathrm{Id}_{2}) : \nabla
v(x') \otimes\eta(x')\Big\rangle ~ \mathrm{d}x ~+ ~ \lim_{h \to
0}\frac{1}{h}\int_{\Omega^1}\left\langle E^h_{13, 23},
  \nabla \eta_3 ({y^h}')\right \rangle ~\mathrm{d}x = 0.
\end{split}
\end{equation}
\end{lemma}
\begin{proof}
{\bf 1.} Let $\eta\in\mathcal{C}_b^2(\mathbb{R}^2,\mathbb{R}^2)$ be
a given test function, and define $\eta_3(x')=-\mbox{div }\eta(x')$.
Given $\theta^h$ as in Lemma \ref{thetah}, with:
\begin{equation}\label{omegah3}
\lim_{h\to 0}\omega_h = \lim_{h\to 0} h\omega_h = +\infty \quad
\mbox{ and } \quad \lim_{h\to 0} h^{1+\frac{1-\gamma}{2}}\omega_h =
0 ~~\mbox{ for some
  fixed } \gamma\in (0,1),
\end{equation}
consider the divergence-free test functions
$\phi^h\in\mathcal{C}^1_{b}(\mathbb{R}^3,\mathbb{R}^3)$:
\begin{equation*}
\phi^h(x', x_3) = \left[\begin{array}{c}
\displaystyle{{\theta^h}'\left(\frac{x_3}{h}\right)
    {\theta^h}\left(\frac{x_3}{h}\right) \eta (x')}\vspace{2mm}\\
\displaystyle{\frac{h}{2}
({\theta^h}\left(\frac{x_3}{h}\right))^2\eta_3
(x')}\end{array}\right].
\end{equation*}
Denoting $\nabla_{tan}$ the gradient in the tangential directions
$e_1, e_2$, we have:
\begin{equation*}
\nabla\phi^h(x', x_3) = \left[\begin{array}{c|c}
\displaystyle{{\theta^h}'\left(\frac{x_3}{h}\right)
    {\theta^h}\left(\frac{x_3}{h}\right) \nabla_{tan}\eta (x')} & \begin{minipage}{7cm}\[\vspace{3mm}
\displaystyle{
\frac{1}{h}\left({\theta^h}''\left(\frac{x_3}{h}\right)
    {\theta^h}\left(\frac{x_3}{h}\right) + ({\theta^h}\left(\frac{x_3}{h}\right) )^2\right)
\eta (x')} \]\end{minipage} \\ \hline \displaystyle{\frac{h}{2}
({\theta^h}\left(\frac{x_3}{h}\right))^2 \nabla_{tan} \eta_3 (x') }&
\begin{minipage}{7cm}\vspace{3mm}\[\displaystyle{
    {\theta^h}'\left(\frac{x_3}{h}\right) {\theta^h}\left(\frac{x_3}{h}\right)
\eta_3 (x')}\]\end{minipage}
\end{array}\right].
\end{equation*}

\medskip

{\bf 2.} Applying now the equilibrium equation (\ref{Eq2}) with
$\phi = \phi^h$, we obtain:
\begin{equation}\label{bigeq4}
\begin{split}
 \int_{\Omega^1}\Big\langle & \Big((\bar R^h)^T R^h E^h  (R^h)^T \bar
R^h\Big)'' - \Big((\bar R^h)^T R^h E^h (R^h)^T \bar
R^h\Big)_{33}\mbox{Id}_2 : {\theta^h}'\left(\frac{y^h_3}{h}\right)
{\theta^h}\left(\frac{y^h_3}{h}\right)
    \nabla_{tan}\eta ({y^h}')\Big\rangle\\
& \qquad\quad + \frac{1}{h} \int_{\Omega^1}  \Big\langle \Big((\bar
R^h)^T R^h E^h (R^h)^T \bar R^h\Big)_{13,23} ,
\left({\theta^h}''\left(\frac{x_3}{h}\right)
    {\theta^h}\left(\frac{x_3}{h}\right) + ({\theta^h}\left(\frac{x_3}{h}\right) )^2\right)
    \eta ({y^h}')\Big\rangle\\
& \qquad\quad + \frac{h}{2} \int_{\Omega^1}  \Big\langle \Big((\bar
R^h)^T R^h E^h (R^h)^T \bar R^h\Big)_{31,32} ,
({\theta^h}\left(\frac{y^h_3}{h}\right))^2
    \nabla_{tan}\eta ({y^h}')\Big\rangle\\
& = h\int_{\Omega^1} \Big\langle f(x') (\bar R^h)_{31, 32} ,
  {\theta^h}'\left(\frac{y_3^h}{h}\right)\theta^h\left(\frac{y^h_3}{h}\right)
\eta({y^h}')\Big\rangle~\mbox{d}x \\
& \qquad\qquad + \frac{h^2}{2} \int_{\Omega^1}  f(x')  (\bar
R^h)_{33} ({\theta^h}\left(\frac{y^h_3}{h}\right))^2 \eta_3
  ({y^h}')~\mbox{d}x.
\end{split}
\end{equation}

In what follows, we will check convergence as $h\to 0$ of each of
the five terms in the identity (\ref{bigeq4}). We first easily
notice that the two terms in the right hand side converge to $0$, as
they are bounded by:
\begin{equation*}
\begin{split}
C\int_{\Omega^1}|f(x')|\Big( h \left|\theta^h
    \left(\frac{y_3^h}{h}\right)\right| + h^2 \left|\theta^h
\left(\frac{y_3^h}{h}\right)\right|^2\Big)~\mbox{d}x
& \leq C \int_{\Omega^1} |f(x')| \left( |y_3^h| + |y_3^h|^2\right)~\mbox{d}x \\
& \leq C\|f\|_{L^2(\Omega^1)} \left( \|y_3^h\|_{L^2(\Omega^1)} +
\|y_3^h\|^2_{L^4(\Omega^1)}\right).
\end{split}
\end{equation*}
Since $\frac{y_3^h}{h} $ has a strong limit in $W^{1, 2}(\Omega^1)$
by (\ref{convi}), it results that $\|y_3^h\|_{L^2}$ and
$\|y_3^h\|_{L^4}$ converge to $0$.

\medskip

{\bf 3.} The third term in (\ref{bigeq4}) is bounded by the
following expression, in view of (\ref{theta_prop}), (\ref{cinque}),
(\ref{convi}) and (\ref{quattro}):
\begin{equation*}
\begin{split}
 C h \int_{\Omega^1} \chi_h |E^h| & (\theta^h\left(\frac{y_3^h}{h}\right))^2 ~\mbox{d}x
+ C h \int_{\Omega^1} (1 - \chi_h) |E^h|
(\theta^h\left(\frac{y_3^h}{h}\right))^2 ~\mbox{d}x
\\ & \leq Ch \int_{\Omega^1}\chi_h |E^h| \left|\frac{y_3^h}{h}\right|^2 ~\mbox{d}x +
Ch\omega^2_h \int_{\Omega^1\setminus B_h } |E^h|~\mbox{d}x \\
&\leq C h \|\chi_h E^h\|_{L^2}
\left\|\frac{y^h_3}{h}\right\|^2_{L^4} + Ch\omega_h^2 h^{2-\gamma}
\leq Ch+ C\left(h^{1 + \frac{1-\gamma}{2}}
  \omega_h\right)^2
\end{split}
\end{equation*}
which converges to $0$ by (\ref{omegah3}).

\medskip

{\bf 4.}  We will now investigate the first term in (\ref{bigeq4}).
Integrated on $\Omega^1\setminus B_h$, it is bounded by:
$$C\omega_h\int_{\Omega^1} (1-\chi_h) |E^h|~\mbox{d}x \leq C\omega_h
h^{2-\gamma} \leq Ch^{1+\frac{1-\gamma}{2}}\omega_h,$$ by
(\ref{quattro}) and hence it converges to $0$ through
(\ref{omegah3}). The same term integrated on $B_h$ equals now the
following sum:
\begin{equation}\label{koko}
\begin{split}
&\int_{\Omega^1} \left({\theta^h}'\left(\frac{y_3^h}{h}\right)
  - 1\right)\theta^h\left(\frac{y_3^h}{h}\right) \cdot
\\ &\qquad \cdot \Big\langle \left((\bar
  R^h)^T R^h \chi_h E^h (R^h)^T \bar R^h\right)''
- \left((\bar R^h)^T R^h \chi_h E^h (R^h)^T \bar R^h\right)_{33}
\mbox{Id}_2 : \nabla_{tan}\eta({y^h}')\Big\rangle~\mbox{d}x\\
&  + \int_{\Omega^1} \theta^h\left(\frac{y_3^h}{h}\right) \cdot \\ &
\qquad \cdot \Big\langle \left((\bar
  R^h)^T R^h \chi_h E^h (R^h)^T \bar R^h\right)''
- \left((\bar R^h)^T R^h \chi_h E^h (R^h)^T \bar R^h\right)_{33}
\mbox{Id}_2 : \nabla_{tan}\eta({y^h}')\Big\rangle~\mbox{d}x.
\end{split}
\end{equation}
The first term in (\ref{koko}) goes to $0$, as it is bounded by:
$$ C \int_{\left\{\frac{|y_3^h|}{h} \geq
    \omega_h\right\}}\left|\frac{y_3^h}{h}\right| |\chi_h E^h|~\mbox{d}x
\leq C
\left\|\frac{y_3^h}{h}\right\|_{L^4(\Omega^1)}\left|\left\{x\in
    \Omega^1; ~ \frac{|y_3^h|}{h}\geq
    \omega_h\right\}\right|^{1/4}\|\chi_h E^h\|_{L^2(\Omega^1)} \leq
\frac{C}{\omega_h^{1/2}}, $$ in view of (\ref{theta_prop}),
(\ref{convi2}), (\ref{cinque}) and recalling (\ref{omegah3}). The
second term of (\ref{koko}) converges to:
\begin{equation}\label{koko2}
\int_{\Omega^1}\Big\langle E'' - E_{33}\mbox{Id}_2 :  (x_3 + v(x'))
  \nabla_{tan}\eta({x}') \Big\rangle ~\mbox{d}x
\end{equation}
because of (\ref{cinque}) and through the following strong
convergences: convergence of $\nabla_{tan}\eta({y^h}')$ to
$\nabla_{tan}\eta (x')$ in $L^5(\Omega^1)$ by (\ref{convi}), of
$(\bar R^h)^T R^h$ to $\mbox{Id}$ in $L^{20}(\Omega)$ by
(\ref{due}), and of $\theta^h\left(\frac{y^h_3}{h}\right)$ to $(x_3
+ v(x'))$ in
  $L^5(\Omega^1)$. The last convergence can be seen from:
\begin{equation*}
\begin{split}
\int_{\Omega^1}\left|\theta^h\left(\frac{y_3^h}{h}\right) - (x_3 +
  v(x'))\right|^5 ~\mbox{d}x& \leq
C\int_{\Omega^1}\left|\theta^h\left(\frac{y_3^h}{h}\right) -
  \frac{y_3^h}{h}\right|^5~\mbox{d}x + C\int_{\Omega^1} \left|\frac{y_3^h}{h}-
  (x_3 + v(x'))\right|^5~\mbox{d}x\\
& \leq C \int_{\left\{\frac{|y_3^h|}{h}\geq
    \omega_h\right\}}\left|\frac{y_3^h}{h}\right|^5~\mbox{d}x + o(1)
 \leq \frac{C}{\omega_h^{1/3}} + o(1) \leq o(1)
\end{split}
\end{equation*}
by (\ref{convi}), (\ref{convi2}) and (\ref{omegah3}). Concluding, we
obtain that the first term in (\ref{bigeq4}) converges to the
expression in (\ref{koko2}).

\medskip

{\bf 5.} Regarding the second term in (\ref{bigeq4}), using
(\ref{tre}), (\ref{theta_prop}), (\ref{en_as2}) and (\ref{convi2})
we note that:
\begin{equation*}
\begin{split}
\Big
|\int_{\Omega^1}&\left({\theta^h}''\left(\frac{y_3^h}{h}\right)\theta^h\left(\frac{y_3^h}{h}\right)
  + {\theta^h}'\left(\frac{y_3^h}{h}\right)^2 - 1\right) \Big\langle
  \left((\bar R^h)^T R^h E^h (R^h)^T \bar R^h\right)_{13, 23},
  \eta({y^h}')\Big\rangle  ~\mbox{d}x\Big|\\
& \qquad \leq \frac{C}{h}\int_{\left\{x\in\Omega^1;
~\frac{|y_3^h(x)|}{h}\geq
    \omega_h\right\}}\left(\frac{1}{\omega_h} \omega_h +
  1\right)|E^h| ~\mbox{d}x\\ & \qquad \leq \frac{C}{h} \int_{\left\{\frac{|y_3^h|}{h}\geq
    \omega_h\right\}} \frac{1}{h^2} W(\nabla u^h(x', hx_3)) + |G^h|
~\mbox{d}x \\ & \qquad \leq \frac{C}{h} \left(h^2 +
\|G^h\|_{L^2(\Omega^1}\left|\left\{x\in\Omega^1; ~
\frac{|y^h(x)|}{h} \geq
      \omega_h\right\}\right|^{1/2}\right) \leq \frac{C}{h}\left(h^2 +
  \frac{1}{\omega_h}\right),
\end{split}
\end{equation*}
which converges to $0$ by (\ref{omegah3}). The remaining part of the
second term in (\ref{bigeq4}) is:
\begin{equation}\label{pupu}
\begin{split}
\frac{1}{h}\int_{\Omega^1}& \Big\langle \left((\bar R^h)^T R^h E^h
(R^h)^T
    \bar R^h\right)_{13, 23}, \eta({y^h}')\Big\rangle ~\mbox{d}x \\ &
= \int_{\Omega^1}\Big\langle \left(A^h E^h (R^h)^T
    \bar R^h\right)_{13, 23}, \eta({y^h}')\Big\rangle ~\mbox{d}x  +
\int_{\Omega^1}\Big\langle \left(E^h (A^h)^T\right)_{13, 23},
  \eta({y^h}')\Big\rangle ~\mbox{d}x \\ &
\qquad + \frac{1}{h} \int_{\Omega^1}\Big\langle (E^h)_{13, 23},
\eta({y^h}')\Big\rangle ~\mbox{d}x,
\end{split}
\end{equation}
where we used the decomposition (\ref{RERA}). Now, exactly as in the
proof of Lemma \ref{lem7.1} and recalling the block structure of the
limiting tensor $A$ in (\ref{CAh}), we see that (\ref{pupu})
converges to:
\begin{equation*}
\begin{split}
&\int_{\Omega^1}\Big\langle \left(AE \right)_{13, 23},
\eta(x')\Big\rangle ~\mbox{d}x  + \int_{\Omega^1}\Big\langle \left(E
A^T\right)_{13, 23},
  \eta(x')\Big\rangle ~\mbox{d}x
+ \frac{1}{h} \int_{\Omega^1}\Big\langle (E^h)_{13, 23},
\eta({y^h}')\Big\rangle ~\mbox{d}x \\ & = \int_{\Omega^1}
\Big\langle \left(E'' -
  E_{33}\mbox{Id}_2\right)\nabla v , \eta(x')\Big\rangle ~\mbox{d}x
+ \frac{1}{h} \int_{\Omega^1}\Big\langle (E^h)_{13, 23},
\eta({y^h}')\Big\rangle ~\mbox{d}x.
\end{split}
\end{equation*}
In conclusion, passing to the limit in (\ref{bigeq4}) clearly yields
(\ref{Lbigeq4}) and achieves the lemma.
\end{proof}

\bigskip

\noindent {\bf Proof of the second Euler-Lagrange equation (\ref{EL2}).}\\
Let now $\xi\in\mathcal{C}^3_b(\mathbb{R}^2,\mathbb{R})$. Applying
Lemma \ref{lem7.1} with $\eta_3 = \xi$, and Lemma \ref{lem7.2} with
$\eta = \nabla\xi$, it follows:
\begin{equation}\label{krak1}
-\int_{\Omega^1}\Big\langle E'' - E_{33}\mbox{Id}_2 : (x_3 +
v(x')\nabla^2 \xi \Big\rangle ~ \mbox{d}x = \bar R_{33} \int_\Omega
f(x') \xi(x') ~ \mbox{d}x'.
\end{equation}
By the first Euler-Lagrange equation in (\ref{krak0}) applied with
$\eta= v\nabla\xi\in W^{2,2}(\Omega,\mathbb{R}^2)$, we see that:
\begin{equation*}
\int_{\Omega^1} \Big\langle E''- E_{33}\mbox{Id}_2 : \nabla v
\otimes \nabla \xi + v(x')\nabla^2 \xi\Big\rangle ~ \mbox{d}x = 0.
\end{equation*}
Thus, (\ref{krak1}) becomes:
\begin{equation*}
\int_{\Omega^1}  \Big\langle E'' - E_{33} \mbox{Id}_2 : \nabla
v\otimes \nabla\xi \Big\rangle ~\mbox{d}x- \int_{\Omega^1}
\Big\langle E'' - E_{33} \mbox{Id}_2 : x_3 \nabla^2 \xi \Big\rangle
~\mbox{d}x = \bar R_{33} \int_\Omega f(x') \xi(x').
\end{equation*}
The equality in  (\ref{EL2}) follows now from the above in view of
(\ref{aiuto}), and by the density of test functions
$\xi\in\mathcal{C}^3_b$ in $W^{2,2}(\Omega,\mathbb{R})$.
\endproof

\end{document}